\documentclass{amsart}

\usepackage{amsthm}
\usepackage{amssymb}
\usepackage{amsfonts}
\usepackage{lscape}
\usepackage{calc,graphicx}

\renewcommand\ge\geqslant
\renewcommand\geq\geqslant
\renewcommand\le\leqslant
\renewcommand\leq\leqslant

\theoremstyle{plain}
\newtheorem{theorem}{Theorem}[section]
\newtheorem{proposition}[theorem]{Proposition}
\newtheorem{lemma}[theorem]{Lemma}

\theoremstyle{remark}
\newtheorem*{remark*}{Remark}

\numberwithin{equation}{section}

\newcommand{\refpart}[1]{({\it #1\/})}

\newcommand{\defeq}{\mathrel{:=}}

\newcommand{\EQN}{E}
\newcommand{\mathrak}[1]{#1} 

\newcommand{\ppt}{\hspace{1pt}} 

\DeclareMathOperator{\Hn}{Hn}

\newcommand{\hpgo}[2]{{}_{#1}{\rm F}_{#2}}

\newcommand{\hpgoppa}[5]{{}_{#1}{\rm F}_{#2}(#3;#4\,|\,#5)}

\newcommand{\hpg}[5]{{}_{#1}{\rm F}_{#2}\!\left(\!\begin{array}{c|}#3\\#4\end{array}\begin{array}{c}#5\end{array}\!\!\right)}

\newcommand{\led}[1]{#1}	
\newcommand{\hpgde}[1]{E(#1)}	
\newcommand{\heunde}[1]{\text{\it HE\ppt}(#1)} 

\DeclareMathOperator{\ord}{ord}

\newcommand{\branch}[3]{#1\,=\,#2\,=\,#3} 
\newcommand{\brep}[2]{$[#1]_{#2}$} 

\newcommand{\comp}[2]{${#2}\cdot{#1}$}
\newcommand{\compp}[3]{${#3}\cdot{#2}\cdot{#1}$}
\newcommand{\nocomposition}{\text{\rm indecomposable}}

\newcommand{\PH}[1]{\text{$H_{#1}$}}

\newcommand{\PHthirtynine}{\PH{39}}
\newcommand{\PHforty}{\PH{40}}
\newcommand{\PHfortyone}{\PH{41}}
\newcommand{\PHfortytwo}{\PH{42}}
\newcommand{\PHfortythree}{\PH{43}}
\newcommand{\PHfortyfour}{\PH{44}}
\newcommand{\PHfortyfive}{\PH{45}}
\newcommand{\PHfortysix}{\PH{46}}
\newcommand{\PHfortyseven}{\PH{47}}
\newcommand{\PHfortyeight}{\PH{48}}

\newcommand{\PN}[1]{\text{$N_{#1}$}} 
\newcommand{\FF}[1]{$F_{#1}$}        
\newcommand{\FE}[1]{$F^{\prime}_{#1}$} 
\newcommand{\FT}[1]{$F^{\prime\prime}_{#1}$} 
\newcommand{\FY}[1]{$F^{*}_{#1}$} 

\newcommand{\GT}[1]{#1}       
\newcommand{\HT}[1]{{#1}_H}   
\newcommand{\CT}[1]{{#1}_C}   
\newcommand{\AT}[1]{{#1}_A}   
\newcommand{\BT}[1]{{#1}_B}   
\newcommand{\DDT}{2\!\times\!2}  
\newcommand{\pback}[1]{\stackrel{\ppt#1}{\longleftarrow}}   

\newcommand{\CC}{\mathbb{C}}

\newcommand{\PP}{\mathbb{P}}
\newcommand{\QQ}{\mathbb{Q}}
\newcommand{\RR}{\mathbb{R}}
\newcommand{\ZZ}{\mathbb{Z}}

\newcommand{\fr}[2]{#1/#2}



\begin{document}

\title[Coverings yielding Heun-to-hypergeometric reductions]{A
  classification of coverings yielding\\ Heun-to-hypergeometric reductions}

\author{Raimundas Vidunas}
\address{Faculty of Mathematics, Kobe University, Rokko-dai 1-1, Nada-ku, 657-8501 Kobe, Japan}
\email{vidunas@math.kobe-u.ac.jp}
\urladdr{www.math.kobe-u.ac.jp/\~{}vidunas}

\author{Galina Filipuk}
\address{Faculty of Mathematics, Informatics, and Mechanics, University of
Warsaw, Banacha 2, 02-097 Warsaw, Poland
657-8501 Kobe, Japan}
\email{filipuk@mimuw.edu.pl}
\urladdr{http://www.mimuw.edu.pl/\~{}filipuk}

\begin{abstract}
  Pull-back transformations between Heun and Gauss hypergeometric equations 
  give useful expressions of Heun functions in terms of better understood
  hypergeometric functions. This article classifies, up~to M\"obius automorphisms,
  the coverings $\PP^1\to\PP^1$ that yield pull-back transformations 
  from hypergeometric to Heun equations with at~least
  one free parameter (excluding
  the cases when the involved hypergeometric equation has cyclic or dihedral monodromy).
  In all, $61$ parametric hypergeometric-to-Heun transformations are found,
  of maximal degree 12. Among them, $28$ pull-backs 
  are compositions of smaller degree transformations between hypergeometric 
  and Heun functions. The $61$ transformations are realized by $48$ different Belyi coverings 
  (though $2$ coverings should be counted twice as their moduli field is quadratic). 
  The same Belyi coverings appear in several other contexts. For example,
  38 of the coverings appear in Herfutner's list of elliptic surfaces over $\PP^1$ 
  with four singular fibers, as their $j$-invariants. In passing, 
  we demonstrate an elegant way to show that there are no coverings $\PP^1\to\PP^1$
  with some branching patterns. 
\end{abstract}


\maketitle

\section{Context and overview}
\label{sec:intro}

The Gauss hypergeometric equation
\begin{equation}
\label{eq:GHE}
\frac{{\rm d}^2y(z)}{{\rm d}z^2}+
\left(\frac{C}{z}+\frac{A+B-C+1}{z-1}\right)\,\frac{{\rm d}y(z)}{{\rm d}z}+\frac{AB}{z\,(z-1)}\,y(z)=0
\end{equation}
and the Heun equation
\begin{equation}
\label{eq:HE}
\frac{{\rm d}^2Y(x)}{{\rm
d}x^2}+\biggl(\frac{c}{x}+\frac{d}{x-1}+\frac{a+b-c-d+1}{x-t}\biggr)\frac{{\rm
d}Y(x)}{{\rm d}x}+\frac{ab\,x-q}{x(x-1)(x-t)} Y(x)=0
\end{equation}
are canonical second-order Fuchsian differential equations on the Riemann sphere~$\PP^1$, 
with $3$ and~$4$ regular singularities, respectively.
Transformations among these equations give identities between their standard 
hypergeometric and Heun solutions. For example, 
there is a single covering $\PP^1\to\PP^1$ of degree 2 (up to M\"obius transformations).
It induces the classical quadratic transformations of hypergeometric functions, such as
\begin{equation}
  \hpg{2}{1}{2A,\,2B\,}{A+B+\tfrac12}{\,x}
  =\hpg{2}{1}{A,\,B\,}{A+B+\tfrac12}{\,4x(1-x)}. 
\end{equation}
Moreover, the same covering induces the well-known Heun-to-Heun quadratic 
transformation \cite[Thm.~4.1]{Maier14}, and an identification of the general
$\hpgoppa{2}{1}{A,B}{C}{4x(1-x)}$ function with a standard local solution  
of Heun's equation 
with the parameters 
$(t,q,a,b,c,d)=\left(\tfrac12,2AB,2A,2B,C,C\right)$. 
These transformations are {\em parametric}, since they have
at least one free parameter such as $A,B$. 

The aim of this paper is classification of all parametric pull-back transformations
between hypergeometric and Heun functions. 
The considered pull-back transformations are of the form
\begin{equation}
\label{eq:algtransf}
z\longmapsto\varphi(x), \qquad y(z)\longmapsto
Y(x)=\theta(x)\,y(\varphi(x)),
\end{equation}
where $\varphi(x)$ is a rational function and $\theta(x)$ is a radical
function, i.e., a product of powers of rational functions.  Geometrically,
transformation~(\ref{eq:algtransf}) \emph{lifts} or \emph{pulls back} a
Fuchsian equation on the curve~$\PP^1_z$ to one on the curve~$\PP^1_x$,
along the covering $\varphi\colon\PP^1_x\to\PP^1_z$.
The {\em gauge prefactor}~$\theta(x)$ is usually chosen 
such that the pulled-back equation has fewer singularities and 
canonical values of some local exponents.

Pull-back transformations between Gauss hypergeometric equations were recently classified 
by Vidunas \cite{Vidunas2009}. Next to the classical quadratic, cubic and Goursat \cite{Goursat1881} 
transformations, a few sets of unpredicted transformations were found, including parametric 
transformations from hypergeometric equations with cyclic or dihedral monodromy.
Moreover, the hypergeometric-to-Heun transformations without the prefactor $\theta(x)$ have been
classified by Maier~\cite{Maier03}. In both classifications, the heart of the problem is
determining the covering maps~$\varphi(x)$ that can appear.  
They are typically  \emph{Belyi maps}, in the sense that 
(apart from dull exceptions of Proposition \ref{prop:nonbelyi} here)
they have at~most $3$ critical values on the Riemann sphere~$\PP_z^1$.  
In~fact, the critical values of those $\varphi(x)$  are typically the 
singular points $z=0$, $z=1$, $z=\infty$ of the hypergeometric equation,
and the branching points include the singularities $x=0$, $x=1$, $x=\infty$
(and $x=t$) of the pulled-back hypergeometric (or Heun) equation. 
The approaches of~\cite{Maier03,Vidunas2009} 
include:
\begin{itemize}
\item[\refpart{i}] determining the \emph{branching patterns} that $\varphi$~can have;
\item[\refpart{ii}] determining which of those patterns can be \emph{realized} by a rational 
function $\varphi(x)$;
\item[\refpart{iii}] normalizing the points $x=0$, $x=1$, $x=\infty$ of $\varphi(x)$,
and deriving identities between hypergeometric and Heun functions 
by identifying corresponding local solutions of thereby related differential equations.
\end{itemize}
This article follows this strategy and the techniques of \cite{Vidunas2009} to
generate a complete list of coverings~$\varphi$ that can appear in parametric
Heun-to-hypergeometric reductions. 
We find 61 different transformations (excluding infinite families of pull-backs 
from hypergeometric equations with cyclic or dihedral monodromy \cite{VidunasHDD}), 
realized by 48 different Belyi coverings.
An explicit formula for each covering is given in Table~\ref{tab:coverings}.
The Belyi maps are not normalized for Step \refpart{iii}. The induced identities
between hypergeometric and Heun functions are thoroughly examined 
in the parallel article \cite{Vidunas2009b}. 
Here we not concerned with the technical issues of determining the prefactor $\theta(x)$,
identifying local solutions, symmetries of the hypergeometric and Heun equations,
nor even introducing Heun functions.

By the Grothendieck correspondence~\cite{Schneps94} any Belyi map 
$\varphi\colon\PP^1_x\to\PP^1_z$ corresponds bijectively to 
a \emph{dessin d'enfant} on $\PP^1_x$, up~to M\"obius isomorphisms 
of both Riemann spheres $\PP^1_z$, $\PP^1_x$. 
Generally, the dessins 
are defined combinatorially as certain bicolored graphs.
For our purposes, the \emph{dessins d'enfant} of a Belyi map $\varphi(x)$ is the graph on $\PP^1_x$ 
obtained as the pre-image of the line segment~$[0,1]$ 
on $\PP_z^1$, up to isotopy. 
The vertices above $z=0$ are colored black, and the vertices above $z=1$ are colored white.
The order of each vertex is equal to the branching order at the corresponding $x$-point.
Figure \ref{fg:dessins} depicts the dessins 
for all 48 encountered Belyi coverings. 
Most of the white points have order 2, and then they are not depicted. Black points of order 3 or 4
are not depicted either, unless they are connected to a white point of order 1. 
A thin edge connects a pair of displayed black and white vertices. 
A thick edge connects two black points (either displayed or clearly branching) 
with an implicit white point  somewhere in the middle. 
Each {\em cell} (i.e., a two-dimensional connected component 
of the complement on $\PP_x^1$, possibly the outer one) represents a point above $z=\infty$. 
The branching order of each cell is determined by counting the number of black points met 
while tracing a loop along its boundary.


\begin{figure}
\[ \begin{picture}(340,348)\thicklines\linethickness{1pt}
\put(0,338){$H_{1}$} \put(37,316){\line(-1,0){22}} \put(15,316){\line(1,0){22}} \qbezier(37,316)(51,302)(51,316) \qbezier(37,316)(51,330)(51,316) \qbezier(15,316)(1,330)(1,316) \qbezier(15,316)(1,302)(1,316) \put(26,316){\line(0,1){11}} \qbezier(26,327)(40,341)(26,341) \qbezier(26,327)(12,341)(26,341) 
\put(61,338){$H_{2}$} \qbezier(74,325)(60,311)(60,325) \qbezier(74,325)(60,339)(60,325) \qbezier(110,325)(124,311)(124,325) \qbezier(110,325)(124,339)(124,325) \put(74,325){\line(1,0){10}} \qbezier(84,325)(92,341)(100,325) \qbezier(84,325)(92,309)(100,325) \put(100,325){\line(1,0){10}} 
\put(132,338){$H_{3}$} \qbezier(164,316)(178,302)(178,316) \qbezier(164,316)(178,330)(178,316) \put(164,316){\line(-1,0){10}} \put(154,316){\line(-1,-1){13}} \put(154,316){\line(-1,1){13}} \put(141,329){\line(0,-1){26}} \qbezier(141,303)(123,316)(141,329) 
\put(185,338){$H_{4}$} \qbezier(200,325)(186,311)(186,325) \qbezier(200,325)(186,339)(186,325) \put(200,325){\line(1,0){10}} \qbezier(210,325)(225,355)(240,325) \qbezier(210,325)(225,295)(240,325) \put(240,325){\line(-1,0){8}} \qbezier(232,325)(220,313)(220,325) \qbezier(232,325)(220,337)(220,325) 
\put(250,338){$H_{5}$} \qbezier(255,303)(242,317)(255,331) \qbezier(255,303)(268,317)(255,331) \qbezier(283,331)(296,317)(283,303) \qbezier(283,331)(270,317)(283,303) \put(283,303){\line(-1,0){28}} \put(255,331){\line(1,0){28}} 
\put(297,338){$H_{6}$} \put(317,338){\line(0,-1){21}} \put(317,317){\line(3,-2){21}} \put(317,317){\line(-3,-2){21}} \put(296,303){\line(1,0){42}} \put(338,303){\line(-3,5){21}} \put(296,303){\line(3,5){21}} 
\put(0,287){$H_{7}$} \put(37,267){\line(-1,0){22}} \put(15,267){\line(1,0){22}} \qbezier(37,267)(52,252)(52,267) \qbezier(37,267)(52,282)(52,267) \qbezier(15,267)(0,282)(0,267) \qbezier(15,267)(0,252)(0,267) \put(26,267){\line(0,1){15}} \put(26,282){\circle*3} 
\put(59,287){$H_{8}$} \qbezier(101,274)(115,260)(115,274) \qbezier(101,274)(115,288)(115,274) \put(101,274){\line(-1,0){10}} \qbezier(91,274)(82,292)(73,274) \qbezier(91,274)(82,256)(73,274) \put(73,274){\line(-1,0){12}} \put(61,274){\circle*3} 
\put(126,287){$H_{9}$} \qbezier(158,267)(172,253)(172,267) \qbezier(158,267)(172,281)(172,267) \put(158,267){\line(-1,0){10}} \qbezier(148,267)(136,291)(124,267) \qbezier(148,267)(136,243)(124,267) \put(124,267){\line(1,0){12}} \put(136,267){\circle*3} 
\put(179,287){$H_{10}$} \put(181,274){\circle*3} \put(181,274){\line(1,0){12}} \put(193,274){\line(1,-1){13}} \put(193,274){\line(1,1){13}} \put(206,287){\line(0,-1){26}} \qbezier(206,261)(224,274)(206,287) 
\put(225,287){$H_{11}$} \put(262,267){\line(-1,0){22}} \put(240,267){\line(1,0){22}} \qbezier(262,267)(277,252)(277,267) \qbezier(262,267)(277,282)(277,267) \qbezier(240,267)(225,282)(225,267) \qbezier(240,267)(225,252)(225,267) \put(251,267){\circle*3} {\thinlines\put(251,267){\line(0,1){14}} \put(251,282){\circle3} }
\put(285,287){$H_{12}$} \qbezier(327,273)(341,259)(341,273) \qbezier(327,273)(341,287)(341,273) \put(327,273){\line(-1,0){10}} \qbezier(317,273)(308,291)(299,273) \qbezier(317,273)(308,255)(299,273) \put(299,273){\circle*3} {\thinlines\put(299,273){\line(-1,0){11}} \put(287,273){\circle3} }
\put(0,237){$H_{13}$} \qbezier(17,222)(1,206)(1,222) \qbezier(17,222)(1,238)(1,222) \put(17,222){\line(1,0){12}} \qbezier(29,222)(43,250)(57,222) \qbezier(29,222)(43,194)(57,222) \put(57,222){\circle*3} {\thinlines\put(57,222){\line(-1,0){12}} \put(44,222){\circle3} }
\put(70,237){$H_{14}$} {\thinlines\put(70,222){\circle3} \put(71,222){\line(1,0){13}} }\put(84,222){\circle*3} \put(84,222){\line(1,-1){15}} \put(84,222){\line(1,1){15}} \put(99,237){\line(0,-1){30}} \qbezier(99,207)(119,222)(99,237) 
\put(126,237){$H_{15}$} \put(151,222){\circle*3} \put(162,222){\line(-1,0){22}} \put(140,222){\line(1,0){22}} \qbezier(162,222)(179,205)(179,222) \qbezier(162,222)(179,239)(179,222) \qbezier(140,222)(123,239)(123,222) \qbezier(140,222)(123,205)(123,222) 
\put(193,237){$H_{16}$} \qbezier(210,222)(194,206)(194,222) \qbezier(210,222)(194,238)(194,222) \put(210,222){\line(1,0){12}} \qbezier(222,222)(234,246)(246,222) \qbezier(222,222)(234,198)(246,222) \put(246,222){\circle*3} 
\put(252,237){$H_{17}$} \put(262,222){\circle*3} \put(262,222){\line(1,-1){15}} \put(262,222){\line(1,1){15}} \put(277,237){\line(0,-1){30}} \qbezier(277,207)(297,222)(277,237) 
\put(299,237){$H_{18}$} \qbezier(316,222)(300,206)(300,222) \qbezier(316,222)(300,238)(300,222) \put(316,222){\line(1,0){13}} \put(329,222){\line(1,-1){11}} \put(329,222){\line(1,1){11}} \put(340,233){\circle*3} \put(340,211){\circle*3} 
\put(0,187){$H_{19}$} \put(3,156){\circle*3} \put(35,188){\circle*3} \put(35,188){\line(-1,-1){10}} \put(3,156){\line(1,1){10}} \qbezier(13,166)(8,183)(25,178) \qbezier(13,166)(30,161)(25,178) 
\put(46,187){$H_{20}$} \put(50,172){\circle*3} \put(50,172){\line(1,0){14}} \qbezier(64,172)(78,200)(92,172) \qbezier(64,172)(78,144)(92,172) \put(92,172){\line(-1,0){14}} \put(78,172){\circle*3} 
\put(105,187){$H_{21}$} \qbezier(122,172)(106,156)(106,172) \qbezier(122,172)(106,188)(106,172) \put(122,172){\line(1,0){13}} \put(135,172){\circle*3} \put(135,172){\line(1,1){11}} \put(146,183){\circle*3} {\thinlines\put(135,172){\line(1,-1){10}} \put(146,161){\circle3} }
\put(160,187){$H_{22}$} \put(163,156){\circle*3} \put(163,156){\line(1,1){10}} \qbezier(173,166)(168,183)(185,178) \qbezier(173,166)(190,161)(185,178) \put(185,178){\circle*3} {\thinlines\put(185,178){\line(1,1){10}} \put(196,189){\circle3} }
\put(209,187){$H_{23}$} \put(213,172){\circle*3} \put(213,172){\line(1,0){14}} \qbezier(227,172)(241,200)(255,172) \qbezier(227,172)(241,144)(255,172) \put(255,172){\circle*3} {\thinlines\put(255,172){\line(-1,0){12}} \put(242,172){\circle3} }
\put(268,187){$H_{24}$} \qbezier(281,169)(258,170)(270,158) \qbezier(281,169)(282,146)(270,158) \put(281,169){\line(1,1){19}} \put(300,188){\circle*3} \put(290,178){\circle*3} 
\put(310,187){$H_{25}$} \put(312,158){\circle*3} \qbezier(312,158)(306,182)(330,176) \qbezier(312,158)(336,152)(330,176) \put(330,176){\line(1,1){12}} \put(342,188){\circle*3} 
\put(0,137){$H_{26}$} \qbezier(17,122)(1,106)(1,122) \qbezier(17,122)(1,138)(1,122) \put(17,122){\line(1,0){13}} \put(30,122){\circle*3} {\thinlines\put(30,122){\line(1,-1){10}} \put(30,122){\line(1,1){10}} \put(41,133){\circle3} \put(41,111){\circle3} }
\put(53,137){$H_{27}$} \put(64,116){\circle*3} \put(76,128){\circle*3} \qbezier(64,116)(59,133)(76,128) \qbezier(64,116)(81,111)(76,128) \put(76,128){\circle*3} {\thinlines\put(64,116){\line(-1,-1){10}} \put(76,128){\line(1,1){10}} \put(87,139){\circle3} \put(53,105){\circle3} }
\put(99,137){$H_{28}$} \put(113,122){\circle*3} \qbezier(113,122)(127,150)(141,122) \qbezier(113,122)(127,94)(141,122) \put(141,122){\circle*3} {\thinlines\put(141,122){\line(-1,0){13}} \put(127,122){\circle3} \put(113,122){\line(-1,0){13}} \put(99,122){\circle3} }
\put(153,137){$H_{29}$} \qbezier(165,118)(142,119)(154,107) \qbezier(165,118)(166,95)(154,107) \put(165,118){\line(1,1){10}} \put(175,128){\circle*3} {\thinlines\put(175,128){\line(1,1){10}} \put(186,139){\circle3} }
\put(194,137){$H_{30}$} \put(196,108){\circle*3} \qbezier(196,108)(190,132)(214,126) \qbezier(196,108)(220,102)(214,126) \put(214,126){\circle*3} {\thinlines\put(214,126){\line(1,1){11}} \put(226,138){\circle3} }
\put(235,137){$H_{31}$} \put(246,122){\circle*3} \qbezier(246,122)(258,146)(270,122) \qbezier(246,122)(258,98)(270,122) \put(270,122){\circle*3} 
\put(277,137){$H_{32}$} \put(282,112){\circle*3} \put(302,132){\circle*3} \put(282,112){\line(1,1){20}} {\thinlines
\put(311,137){$H_{33}$} \put(313,122){\circle3} \put(314,122){\line(1,0){14}} \put(328,122){\circle*3} \put(328,122){\line(1,-1){11}} \put(328,122){\line(1,1){11}} \put(340,134){\circle3} \put(340,110){\circle3} }
\put(0,87){$H_{34}$} \put(4,58){\circle*3} \put(4,58){\line(1,1){14}} \put(18,72){\circle*3} {\thinlines\put(18,72){\line(-1,-1){13}} \put(18,72){\line(1,1){13}} \put(32,86){\circle3} \put(4,58){\circle3} }
\put(42,87){$H_{35}$} \put(45,58){\circle*3} \put(73,86){\circle*3} \put(45,58){\line(1,1){28}} \put(59,72){\circle*3} 
\put(82,87){$H_{36}$} \put(84,72){\circle*3} \put(84,72){\line(1,0){15}} \put(99,72){\circle*3} {\thinlines\put(99,72){\line(1,-1){11}} \put(99,72){\line(1,1){11}} \put(111,84){\circle3} \put(111,60){\circle3} }
\put(122,87){$H_{37}$} {\thinlines\put(124,72){\circle3} \put(125,72){\line(1,0){14}} }\put(139,72){\circle*3} \put(139,72){\line(1,-1){12}} \put(139,72){\line(1,1){12}} \put(151,84){\circle*3} \put(151,60){\circle*3} 
\put(162,87){$H_{38}$} \put(164,72){\circle*3} \put(164,72){\line(1,0){15}} \put(179,72){\line(1,-1){12}} \put(179,72){\line(1,1){12}} \put(191,84){\circle*3} \put(191,60){\circle*3} 
\put(205,87){$H_{39}$} \put(206,56){\circle*3} \put(238,88){\circle*3} \put(206,56){\line(1,1){33}} \put(227,77){\circle*3} \put(217,67){\circle*3} 
\put(248,87){$H_{40}$} \put(251,72){\circle*3} \put(277,72){\circle*3} \put(251,72){\line(1,0){26}} \qbezier(264,72)(278,100)(292,72) \qbezier(264,72)(278,44)(292,72) \put(292,72){\circle*3} 
\put(298,87){$H_{41}$} \put(310,72){\circle*3} \put(340,72){\circle*3} \put(340,72){\line(-1,-1){15}} \put(310,72){\line(1,1){15}} \put(325,87){\circle*3} \put(325,57){\circle*3} \put(325,57){\line(-1,1){15}} \put(325,87){\line(1,-1){15}} 
\put(0,37){$H_{42}$} \qbezier(22,22)(-2,22)(10,10) \qbezier(22,22)(22,-2)(10,10) \put(22,22){\line(0,1){16}} \put(22,22){\line(1,0){16}} \put(38,22){\circle*3} \put(22,38){\circle*3} 
\put(53,37){$H_{43}$} \put(86,22){\line(-1,0){16}} \put(70,22){\line(1,0){16}} \qbezier(86,22)(104,4)(104,22) \qbezier(86,22)(104,40)(104,22) \qbezier(70,22)(52,40)(52,22) \qbezier(70,22)(52,4)(52,22) 
\put(113,37){$H_{44}$} \put(136,22){\circle*3} \qbezier(136,22)(112,22)(124,10) \qbezier(136,22)(136,-2)(124,10) \put(136,22){\line(0,1){16}} \put(136,38){\circle*3} {\thinlines\put(136,22){\line(1,0){15}} \put(152,22){\circle3} }
\put(160,37){$H_{45}$} \put(173,17){\circle*3} \put(162,6){\circle*3} \put(184,28){\circle*3} \put(162,6){\line(1,1){22}} {\thinlines\put(184,28){\line(1,1){10}} \put(195,39){\circle3} }{\thinlines
\put(206,37){$H_{46}$} \put(213,5){\circle*3} \put(213,5){\line(1,1){10}} \put(224,16){\circle3} \qbezier(224,17)(219,34)(235,28) \qbezier(224,17)(241,12)(235,28) \put(236,28){\circle*3} \put(236,28){\line(1,1){10}} \put(247,39){\circle3} }
\put(259,37){$H_{47}$} \qbezier(276,22)(250,22)(263,9) \qbezier(276,22)(276,-4)(263,9) \put(276,22){\line(1,1){13}} \put(289,35){\circle*3} 
\put(299,37){$H_{48}$} \put(324,22){\circle*3} {\thinlines\put(324,22){\line(0,1){15}} \put(324,22){\line(0,-1){15}} \put(324,6){\circle3} \put(324,38){\circle3} \put(324,22){\line(-1,0){15}} \put(324,22){\line(1,0){15}} \put(340,22){\circle3} \put(308,22){\circle3} }
\end{picture} \]
\caption{{\em Dessins d'enfant} of the Belyi coverings for parametric Heun-to-hypergeometric reductions}
\label{fg:dessins}
\end{figure}
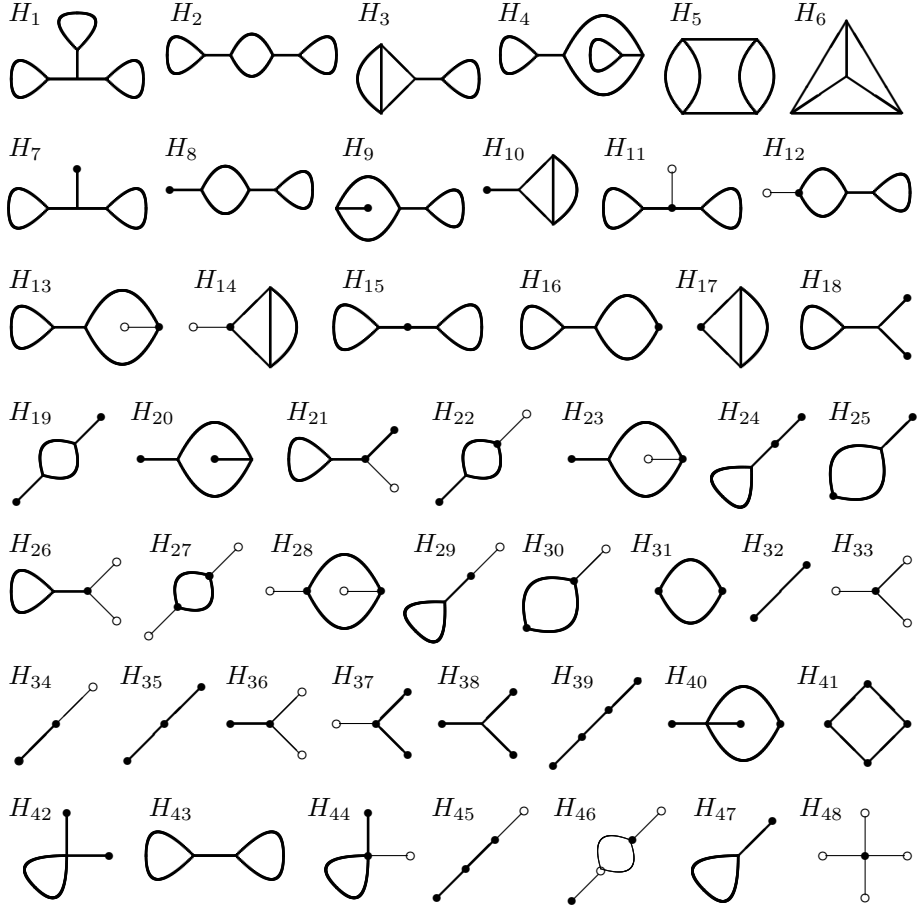


It is instructive to follow the branching orders and incidences 
on  the dessins 
while following our classification of possible coverings in Tables \ref{tab:clas}--\ref{tab:clas2}.
In~principle, the pull-back Belyi coverings can be classified by generating and counting 
the dessins  satisfying the suitable branching patterns. However, it is difficult to ensure 
completeness  of a large list of dessins. We first computed the Belyi coverings explicitly,
then easily generated the required dessins by combinatorial consideration. 
For each possible branching pattern, there is at most one Belyi covering except for the coverings
$H_{21}$ and $H_{44}$. 
Therefore completeness and identification of the dessins is quickly established. 
The coverings $H_{21}$, $H_{44}$ 
are defined over $\QQ(\sqrt{-3})$ and $\QQ(i)$, respectively. 
All other coverings are defined over $\QQ$ and $\RR$, hence their dessins have a reflection symmetry. 
The dessins for $H_{21}$, $H_{44}$ should actually be counted twice,
as the complex conjugation gives non-isotopic dessins. 
The proper count of dessins and Belyi coverings is therefore 50, not 48.

Many of the encountered Belyi coverings occur in other contexts, particularly
in the theory of elliptic surfaces and Picard-Fuchs equations.
The coverings from $H_1$ to $H_{38}$ occur in
Herfurtner's list~\cite{Herfurtner91} of elliptic surfaces with four singular fibers,
up to M\"obius transformations. The order of these coverings follows \cite[Table 3]{Herfurtner91},
and the numbering is used in \cite{Movasati2009} where the corresponding pull-backs 
to Heun equations (specializable to Picard-Fuchs equations for the elliptic surfaces) are observed. 
The coverings $H_1$ to $H_6$ have the maximal degree 12,
and produce the Beauville list \cite{Beauville82} of the coverings
generating {\em semi-stable} elliptic surfaces with four singular fibers. 
Their branching orders above $z=0$ are all 3,
and above $z=1$ they are all 2, as can be seen from the dessins.
The branching pattern of $H_1$ is written by us as follows:
\begin{equation}
\label{eq:firstpattern}
\text{\branch{\brep26}{\brep34}{9+1+1+1}.}
\end{equation}
The four singular fibers of the corresponding elliptic surface have the Kodaira types $I_9,I_1,I_1,I_1$. 
This covering is also described as a {\em Davenport-Stothers triple} \cite{ShiodaDS}:
it can be written as $F^3/G^2$, 
where $F,G$ are polynomials of degree 4 and 6 (respectively), 
such that  the polynomial $F^3-G^2$ has the minimal possible degree 3.

A pull-back transformation 
defined over $\RR$ can be nicely illustrated
by subdivisions of the Schwarz quadrangle for the pulled-back Heun equation
into Schwarz triangles for the initial hypergeometric equations, following 
\cite{Hodgkinson18,Hodgkinson20}. In the hyperbolic geometry setting,
these are  \emph{Coxeter decompositions} \cite{Felikson98} or {\em divisible tilings}
\cite{Broughton2000} of a hyperbolic quadrangle into mutually similar hyperbolic triangles. 
We describe these picturesque illustrations in \S \ref{subsec:Ftilings} and Figure \ref{fg:felixon}.

This article is structured as follows.   Section~\ref{sec:3} establishes pivotal lemmas
on the behavior of singularities and local exponents of Fuchsian equations
under pull-back transformations.
Section~\ref{sec:4} presents the main results in Tables
\ref{tab:clas}--\ref{tab:coverings}, and explains them (and the notation) in a few steps.
Of the three mentioned generation steps \refpart{i}--\refpart{iii},
the first step is elaborated in \S\S \ref{subsec:step2}, \ref{subsec:step3},
while computations for Step \refpart{ii} are reviewed in \S \ref{subsec:step4}.
Step \refpart{iii} is thoroughly considered in the parallel paper \cite{Vidunas2009b}.
Furthermore, \S \ref{sec:5} relates our classification to Herfurtner's list ~\cite{Herfurtner91}
and Felikson's list of Coxeter decompositions \cite{Felikson98},
and \S \ref{subsec:composite} examines the composite transformations. 
Section~\ref{sec:6} presents an elegant approach to prove non-existence 
(or uniqueness) 
of Belyi coverings with some branching patterns,
and applies it not only to the obtained list of branching patterns,
but also to  the Miranda--Persson classification~\cite{Miranda89} 
of K3 semi-stable elliptic surfaces with six singular fibers. 


\section{Pull-backs and local exponents}
\label{sec:3}

The singular points and the local exponents of Gauss hypergeometric equation (\ref{eq:GHE})
are usefully encoded in the Riemann P-symbol scheme
\begin{equation}
\label{eq:PsymbolGHE}
P\left
\{\begin{array}{ccc|c}
0 & 1 & \infty & z\\ \hline
0 & 0 & a & \\
1-c & c-a-b & b & 
\end{array} 
\right\}.
\end{equation}
The local exponent differences at the 3 singular points are therefore
\begin{equation} \label{eq:hpgled}
1-c, \quad c-a-b, \quad a-b. 
\end{equation}
Similarly, the Riemann scheme of the Heun equation (\ref{eq:HE}) is
\begin{equation}
\label{eq:PsymbolHE}
P\left\{
\begin{array}{cccc|c}
0 & 1 &t & \infty &x\\ \hline
0 & 0 &0 & a & \\
1-c & 1-d & c+d-a-b & b &
\end{array}
\right\}.
\end{equation}
The parameters $a,b,c,d$ determine the local exponents, 
while the parameter~$q$ is \emph{accessory}.  
In particular, the 4 exponent differences are 
\begin{equation}
1-c, \quad 1-d, \quad c+d-a-b, \quad a-b.
\end{equation}
The Heun equation contains many interesting special cases, including the Lam\'e
equation~\cite{Erdelyi53}. 
The Heun equation and its solutions appear in problems of diffusion, wave
propagation, heat and mass transfer, magneto-hydrodynamics, particle
physics, and the cosmology of the very early universe.

By $\hpgde{\alpha,\beta,\gamma}$ we denote a Gauss hypergeometric
equation of the form~(\ref{eq:GHE}) with the exponent differences (\ref{eq:hpgled})
equal to $\alpha,\,\beta,\,\gamma$ in some order.  
Similarly, by $\heunde{\alpha,\beta,\gamma,\delta}$ we denote a Heun equation
of the form~(\ref{eq:HE}) with its exponent differences equal to
$\alpha,\,\beta,\,\gamma,\,\delta$ in some order. 
These notations do not assign local exponents to particular singular points,
nor they specify the accessory parameter~$q$.

The degree of a pull-back transformation (\ref{eq:algtransf}) between Fuchsian 
equations is the degree of the rational function $\varphi(x)$. 
The existence of a pull-back 
from some $\hpgde{\alpha_1,\beta_1,\gamma_1}$ to some
$\heunde{\alpha_2,\beta_2,\gamma_2,\delta_2}$ of degree~$D$
will be indicated by
\begin{equation}
\label{eq:GHEtoHE}
\hpgde{\alpha_1,\beta_1,\gamma_1}\pback{D}\heunde{\alpha_2,\beta_2,\gamma_2,\delta_2}.  
\end{equation}
Sometimes the pull-back covering or the transformation will be indicated more specifically 
by a subscript on the degree~$D$.  Similarly,
\begin{equation*}
\label{eq:GHEtoGHE}
\hpgde{\alpha_1,\beta_1,\gamma_1}\pback{D}\hpgde{\alpha_2,\beta_2,\gamma_2}, \quad
\heunde{\alpha_1,\beta_1,\gamma_1,\delta_2}\pback{\;D_H}\heunde{\alpha_2,\beta_2,\gamma_2,\delta_2}
\end{equation*}
will indicate pull-back transformations between hypergeometric or between Heun equations. 
For brevity, we refer to these three types of  transformations as 
{\em Gauss-to-Heun},  {\em Gauss-to-Gauss} (or just {\em hypergeometric})  
and {\em Heun-to-Heun} pull-back transformations. In particular, 
the 3 quadratic transformations mentioned at the beginning of this article actually are:
 \begin{align}
\label{eq:quadtr1}
& \hpgde{\fr12,\,\alpha,\,\beta} \pback{2}\hpgde{\alpha,\,\alpha,\,2\beta},\\ \label{eq:quadtr3}
& \heunde{1/2,\,1/2,\,\alpha,\,\beta}\pback{\,\HT2}\heunde{\alpha,\,\alpha,\,\beta,\,\beta},\\ \label{eq:quadtr2}
& \hpgde{\alpha,\,\beta,\,\gamma}\pback{2}\heunde{\alpha,\,\alpha,\,2\beta,\,2\gamma}. 
\end{align}
As in the notation $(\alpha_1,\beta_1,\gamma_1)\pback{D}(\alpha_2,\beta_2,\gamma_2)$
of~\cite{Vidunas2009}, the arrows follow the direction of the covering
$\varphi\colon\PP_x^1\to\PP^1_z$.  To~emphasize: these notations
indicate the existence of \emph{some} differential equations with the stated exponent 
differences that are related by a pull-back transformation, rather than the existence
of a pull-back between \emph{any} equations with the
specified exponent differences.

Our classification is obtained by considering the behavior of singularities and
local exponents of Fuchsian equations under pull-backs. 
Any transformation of the form~(\ref{eq:algtransf}) pulls-back a Fuchsian
equation to a Fuchsian equation, usually with more singular points.
To pull-back a hypergeometric equation to a Fuchsian equation with just
4 singular points, special restrictions apply to the covering $\varphi(x)$ and
the hypergeometric equation. 

The following definitions are taken from~\cite{Vidunas2009}.  
An \emph{irrelevant singular point} of a Fuchsian equation
is a non-logarithmic singular point where the local exponent difference is equal to~$1$.  
For comparison, an {\em ordinary} (i.e., non-singular) point 
is a non-logarithmic point with the local exponents $0$ and $1$,
and an {\em apparent singularity} is a non-logarithmic singular point
with the local exponents $0$ and an integer $k>1$.
A~\emph{relevant singular point} is one that is not irrelevant.  
Any irrelevant singular point can be turned into an ordinary point by
a pull-back~(\ref{eq:algtransf}) which is prefactor-only, 
i.e., one with $\varphi(x)=x$.  Hence, what is of primary importance is how many
\emph{relevant} singular points the pulled-back equation has.  This number
is affected only by the choice of covering $\varphi(x)$, and not by the
choice of prefactor~$\theta(x)$.

The following two lemmas describe the crucial behavior of singularities
and local exponents under pull-backs.  
\begin{lemma}
\label{lem:genrami}
Let\/ $\varphi\colon\PP^1_x\to\PP^1_z$ be a finite covering.  Let\/ $\EQN_1$
denote a Fuchsian equation on\/ $\PP^1_z$, and let\/ $\EQN_2$ denote
the pull-back on\/ $\PP^1_x$ of\/ $\EQN_1$ by\/ transformation {\rm(\ref{eq:algtransf})}. 
For any\/ $S\in\PP^1_x$, let\/ 
$k:=\ord_\varphi(P)$ denote the branching order of\/ $\varphi$ at\/ $S$.
\begin{enumerate}
\item The exponents of\/ $\EQN_2$ at\/ $S$ equal\/ $k\alpha_1+\gamma$,
  $k\alpha_2+\gamma$, where:
\begin{itemize}
\item $\alpha_1,\alpha_2$ are the exponents of\/ $\EQN_1$ at\/
$\varphi(S)\in\PP^1_z$;
\item $\gamma$ is the exponent of the radical function\/ $\theta(x)$ at\/
  $S$.
\end{itemize}
\item If\/ $\varphi(S)$ is an ordinary point of\/ $\EQN_1$, then\/ $S$ will
  fail to be a relevant singular point for\/ $\EQN_2$ if and
  only if\/ $k=1$ {\rm(}i.e., the covering\/ $\varphi$ does not branch at\/
  $S$, i.e., $S$ is not a branching point of\/ $\varphi${\rm)}.
\item If $\varphi(S)$ is a singular point of $\EQN_1$, then $S$ will fail
  to be a relevant singular point of\/ $\EQN_2$ if and only if 
\begin{itemize}
\item $k>1$ and the exponent difference at\/ $\varphi(S)$ is equal to\/
  $1/k$; or,
\item $k=1$ and\/ $\varphi(S)$ is irrelevant.
\end{itemize}
 In either case\/ $S$ will be an irrelevant singular point or an ordinary
 point.
\end{enumerate}
\end{lemma}

\begin{proof}
The first statement is mentioned in the proof
of~\cite[Lemma~2.4]{Vidunas2009}.  The other two statements are parts 2
and~3 of~\cite[Lemma~2.4]{Vidunas2009}.
\end{proof}

\begin{lemma}
\label{lem:genrami2}
Let\/ $\varphi\colon\PP^1_x\to\PP^1_z$ be a covering of degree\/ $D$, and
let\/ $\Delta$ denote a set of\/ $3$ points on\/ $\PP^1_z$.  
\begin{enumerate}
\item If all branching points of\/ $\varphi$ lie above\/ $\Delta$, i.e., no point
of\/ $\PP^1_z\setminus\Delta$ is a critical value of\/ $\varphi$,
then there are exactly\/ $D+2$ distinct points on\/ $\PP^1_x$ above\/ $\Delta$.
Otherwise, there are more than\/ $D+2$ distinct points above\/ $\Delta$.
\item If there are exactly $D+3$ distinct points above $\Delta$, there is only one 
branching point that is not above $\Delta$.
\end{enumerate}
\end{lemma}

\begin{proof}
The first statement is part~1 of~\cite[Lemma~2.5]{Vidunas2009}.  It follows from the
Hurwitz formula \cite[Corollary IV.2.4]{Hartshorne77}, which says that 
the sum of $\ord_\varphi(P)-1$ over the branching points~$P\in\PP^1_x$ must equal $2(D-1)$.
The second statement is a slight extension (utilized in \cite{Kitaev2005}).
\end{proof}

Suppose one starts with a hypergeometric equation $\EQN_1$ on~$\PP_z^1$.
Let $\Delta$ denote the set $\{0,\,1,\,\infty\}$ containing the singularities of $\EQN_1$. 
 It follows from the above lemmas that to minimize the number of singular
points of a pull-back of~$\EQN_1$, 
one should typically allow branching points of~$\varphi$ only above~$\Delta$.
Otherwise, there would be more than $D+2$ distinct points above~$\Delta$,
and generically, each of these $D+2$ points would be a singular point of
the pulled-back equation.  By Lemma~\ref{lem:genrami}\refpart{c}, further
minimization is possible if one or more of the exponent differences
of~$\EQN_1$ in~$\Delta$ are restricted to be of the form~$1/k$.

Recall that a covering $\varphi\colon\PP^1\to\PP^1$ is a \emph{Belyi
  covering}~\cite{Shabat2000} if it is unbranched above the complement of a
set of three points, such as $\{0,1,\infty\}$.  By the above consideration,
one expects that the pull-back coverings for Gauss-to-Heun transformations
will typically be Belyi coverings.  The following proposition classifies
the rather degenerate situations in which non-Belyi coverings can occur.

\begin{proposition}
\label{prop:nonbelyi}
Suppose there is a pull-back transformation 
{\rm(\ref{eq:algtransf})} of a hypergeometric equation\/ $\EQN_1$ 
to a Fuchsian equation with at most\/ $4$ singular points, 
and the covering defined by the rational function $\varphi(x)$ is not a Belyi map. 
Then one of the following statements must hold:
\begin{itemize}
\item[(i)] Two of the three exponent differences of $\EQN_1$ are equal to $1/2$; or
\item[(ii)] $\EQN_1$ has a basis of solutions consisting of algebraic functions of $z$.
\end{itemize}
\end{proposition}
\begin{proof}
Let $D=\deg\varphi$, and $\Delta=\{0,\,1,\,\infty\}\subset\PP^1_z$.  
Since $\varphi\colon\PP^1_x\to\PP^1_z$ is not a Belyi map,
there is a branching point $P_0$ that does \emph{not} lie above $\Delta$.  
By part \refpart{a} of Lemma~\ref{lem:genrami2}, there are at least $D+3$
distinct points above $\Delta$.  At~most $3$~of them
can be singularities of the pulled-back equation, because $P_0$
will be a singularity by Lemma~\ref{lem:genrami}\refpart{b}. 
Therefore there are at~least $D$ ordinary points above $\Delta$.

One or more of the 3 exponent differences of $\EQN_1$ must be of the form~$1/k$
for an integer $k\ge1$, because only then ordinary points occur above $\Delta$
by Lemma~\ref{lem:genrami}\refpart{c}.
Above a point of $\Delta$ with the exponent difference~$1/k$, there may be at~most $D/k$ ordinary points.  
Let $M$ denote the number of restricted exponent differences of $\EQN_1$. 
There are three possibilities:
\begin{itemize}
\item $M=1$. One must have $k=1$, and by Lemma~\ref{lem:genrami}\refpart{c}, 
this point is not a relevant singularity for ~$\EQN_1$. Let $m$ denote the number
of distinct points above the two (generally) relevant singularities of $\EQN_1$.
If $m=2$, the covering is cyclic (i.e., M\"obius-equivalent to $\varphi(x)=x^D$). 
If $m=3$, there is only one branching point not above the relevant singularities of $\EQN_1$,
by Lemma  \ref{lem:genrami2}\refpart{b} basically. Hence $\varphi$ is a Belyi covering
for $m\le 3$. If $m>3$, the pulled-back equation will have more than 4 singularities.
\item $M=2$.  The exponent differences will be $1/k$, $1/\ell$ with
  $k,\ell$ positive integers and $D/k+D/\ell\ge D$.  One must have
  $1/k,1/\ell=1/2$, which is case~(i).
\item $M=3$.  The exponent differences will be $1/k$, $1/\ell$,
  $1/m$ with $k,\ell,m$ positive integers and $D/k+D/\ell+D/m\ge D$, i.e.,
  $1/k+1/\ell+1/m\ge 1$.  The subcase $1/k+1/\ell+1/m=1$ can be ruled~out,
  since even if the points above $\Delta$ are optimally
  arranged, there will be fewer than $D$ ordinary points above
  $\Delta$, contradicting Lemma~\ref{lem:genrami2}\refpart{a}.  It is
  known ~\cite{Erdelyi53,Poole36} that in the subcase $1/k+1/\ell+1/m>1$, 
  the equation $\EQN_1$ has only algebraic solutions.\qedhere
\end{itemize}
\end{proof}

\begin{remark*}
In case (i), the projective monodromy group of 
$\EQN_1$ is generally an infinite dihedral group. As we recall in \S \ref{sec:6},
the possible projective monodromies in case (ii) are: a finite cyclic, a finite dihedral,
$\mathrak{A}_4$~(tetrahedral), $\mathrak{S}_4$~(octahedral) or $\mathrak{A}_5$~(icosahedral) groups.
If $M=1$, the monodromy is generally an infinite cyclic group. 
Gauss-to-Heun transformations 
with (finite or infinite)  cyclic or dihedral monodromy 
are going to be considered throughly in a separate article \cite{VidunasHDD}.
\end{remark*}

\section{Main result: Generation and classification}
\label{sec:4}

Here we present the method and the results of classification of
Gauss-to-Heun transformations with at least one free parameter.
Following part \refpart{c} of Lemma \ref{lem:genrami}, we restrict $m\in\{0,1,2\}$
local exponent differences of the general hypergeometric equation (\ref{eq:GHE}) to
the reciprocals of integers $k\ge 1$. Thereby we have $M=3-m$ free parameters.
Basically, the free parameters are the unrestricted exponent differences.

We ignore the cases (considered in \cite{VidunasHDD})
when an exponent difference is restricted to 1 at a non-logarithmic singularity,
or when two exponent differences are restricted to $1/2$, as then the hypergeometric
equation has cyclic or dihedral monodromy. 
Apart from this, Tables \ref{tab:clas}, \ref{tab:clas1}, \ref{tab:clas2} 
below give a full list of  Gauss-to-Heun pull-back  transformations with a free parameter 
in terms of the exponent differences (in the first two columns), the degree and
the branching pattern of the pull-back covering (in the next two columns)
among the entries where a covering is indicated by the $H$-notation in the last column. 
Table {\rm\ref{tab:coverings}} gives a full list of the encountered Belyi maps
(up to M\"obius transformations and complex conjugation), 
and the introductory Figure \ref{fg:dessins} depicts the {\em dessins d'enfant} of those Belyi maps.
The parallel article \cite{Vidunas2009b} identifies the pulled-back Heun equations 
in detail, and gives a representative list of transformation formulas between 
hypergeometric and Heun functions. 

\begin{table}
\begin{center}
{
\small
\begin{tabular}{llcll}
\hline
\multicolumn{2}{@{}c}{Exponent differences} & Deg. &
Branching pattern & Covering characterization, \\
\cline{1-2} hyperg. & Heun & $D$ & above singularities & composition \\
\hline
$\led{\alpha,\,\beta,\,\gamma}$ & $\led{\alpha,\,\alpha,\,2\alpha,\,2\gamma}$ & 2
& \branch{2}{1+1}{1+1}
& \PH{32}, \FF{1}, \FY{1}, \nocomposition \\
\hline 
$\led{\fr12,\,\alpha,\,\beta}$
& $\led{\alpha,\,\alpha,\,2\alpha,\,4\beta}$ & 4
& \branch{\brep22}{4}{2+1+1}
& \PH{35}, \FF{4}, \FY{5}, \comp{\GT2}{2} \\
& $\led{\alpha,\,3\alpha,\,\beta,\,3\beta}$ & 
& \branch{\brep22}{3+1}{3+1}
& \PHfortyseven, \FE{4}, \FY{6}, \nocomposition \\
& $\led{2\alpha,\,2\alpha,\,\beta,\,3\beta}$ & 
& \branch{\brep22}{3+1}{2+2}
& no covering, \PN{27} \\
& $\led{2\alpha,\,2\alpha,\,2\beta,\,2\beta}$ & 
& \branch{\brep22}{2+2}{2+2}
& \PH{31}, \FF{3}, \FY{4}, $2\times 2$ \\
& $\led{1/2,\,\alpha,\,2\alpha,\,3\beta}$ & 3
& \branch{\brep21+1}{2+1}{3}
& \PH{34}, \FF{2}, \FY{2}, \nocomposition \\ 
\hline 
$\led{1/3,\,\alpha,\,\beta}$
& $\led{\alpha,\,2\alpha,\,\beta,\,2\beta}$ & 3
& \branch{\brep31}{2+1}{2+1}
& \PH{34}, \FT{2}, \FY{3}, \nocomposition \\
& $\led{\alpha,\,\alpha,\,\alpha,\,3\beta}$ & 
& \branch{\brep31}{3}{1+1+1} & \PH{33}, \nocomposition \\
\hline
\end{tabular} 
}
\bigskip
\caption{Possible branching patterns of hypergeometric-to-Heun
  transformations with 2 or 3 free parameters.}
\label{tab:clas} 
\end{center}
\end{table}

\begin{table}
\begin{center}
{
\small
\begin{tabular}{lcll} 
\hline
Exponent differences & Deg. & Branching pattern & Covering characterization, \\
of the Heun equation & $D$ & above singularities & composition \\
\hline
$\led{\alpha,\,\alpha,\,\alpha,\,9\alpha}$ & 12
& \branch{\brep26}{\brep34}{9+1+1+1}
& \PH{1}, \comp{\GT4}{\CT3} \\
$\led{\alpha,\,\alpha,\,2\alpha,\,8\alpha}$ & 
& \branch{\brep26}{\brep34}{8+2+1+1}
& \PH{2},  \FF{23}, \FY{34}, \compp{\GT3}{\GT2}{2}\\
$\led{\alpha,\,\alpha,\,3\alpha,\,7\alpha}$ & 
& \branch{\brep26}{\brep34}{7+3+1+1}
& no covering, \PN{1} \\
$\led{\alpha,\,2\alpha,\,2\alpha,\,7\alpha}$ & 
& \branch{\brep26}{\brep34}{7+2+2+1}
& no covering, \PN{2} \\
$\led{\alpha,\,\alpha,\,4\alpha,\,6\alpha}$ & 
& \branch{\brep26}{\brep34}{6+4+1+1}
& no covering, \PN{3} \\
$\led{\alpha,\,2\alpha,\,3\alpha,\,6\alpha}$ & 
& \branch{\brep26}{\brep34}{6+3+2+1}
& \PH{3}, \FF{27}, \FY{33}, \comp{\GT4}{3}, \comp{\GT3}{4}\\
$\led{2\alpha,\,2\alpha,\,2\alpha,\,6\alpha}$ & 
& \branch{\brep26}{\brep34}{6+2+2+2}
& no covering, \PN{4} \\
$\led{\alpha,\,\alpha,\,5\alpha,\,5\alpha}$ & 
& \branch{\brep26}{\brep34}{5+5+1+1}
& \PH{4},  \FF{24},  \FY{32}, \comp{6}{\HT2} \\
$\led{\alpha,\,2\alpha,\,4\alpha,\,5\alpha}$ & 
& \branch{\brep26}{\brep34}{5+4+2+1}
& no covering, \PN{5} \\
$\led{\alpha,\,3\alpha,\,3\alpha,\,5\alpha}$ & 
& \branch{\brep26}{\brep34}{5+3+3+1}
& no covering, \PN{6} \\
$\led{2\alpha,\,2\alpha,\,3\alpha,\,5\alpha}$ & 
& \branch{\brep26}{\brep34}{5+3+2+2}
& no covering, \PN{7} \\
$\led{\alpha,\,3\alpha,\,4\alpha,\,4\alpha}$ & 
& \branch{\brep26}{\brep34}{4+4+3+1}
& no covering, \PN{8} \\
$\led{2\alpha,\,2\alpha,\,4\alpha,\,4\alpha}$ & 
& \branch{\brep26}{\brep34}{4+4+2+2}
& \PH{5},  \FF{22},  \FY{31}, 
\compp{\GT2}{\GT{\CT3}}{2},
\comp{\GT3}{\DDT} \\
$\led{2\alpha,\,3\alpha,\,3\alpha,\,4\alpha}$ & 
& \branch{\brep26}{\brep34}{4+3+3+2}
& no covering, \PN{9} \\
$\led{3\alpha,\,3\alpha,\,3\alpha,\,3\alpha}$ & 
& \branch{\brep26}{\brep34}{3+3+3+3}
& \PH{6}, \comp{\GT4}{\CT3}, \compp{\CT3}{\HT2}{\HT2} \\ 
$\led{\fr13,\,\alpha,\,\alpha,\,8\alpha}$ & 10
& \branch{\brep25}{\brep33+1}{8+1+1}
& \PH{7}, \nocomposition \\
$\led{\fr13,\,\alpha,\,2\alpha,\,7\alpha}$ & 
& \branch{\brep25}{\brep33+1}{7+2+1}
& \PH{8},  \FF{21}, \FY{28}, \nocomposition \\
$\led{\fr13,\,\alpha,\,3\alpha,\,6\alpha}$ & 
& \branch{\brep25}{\brep33+1}{6+3+1}
& no covering, \PN{10} \\
$\led{\fr13,\,2\alpha,\,2\alpha,\,6\alpha}$ & 
& \branch{\brep25}{\brep33+1}{6+2+2}
& no covering, \PN{11} \\
$\led{\fr13,\,\alpha,\,4\alpha,\,5\alpha}$ & 
& \branch{\brep25}{\brep33+1}{5+4+1}
& \PH{9},  \FF{19}, \FY{29}, \nocomposition \\
$\led{\fr13,\,2\alpha,\,3\alpha,\,5\alpha}$ & 
& \branch{\brep25}{\brep33+1}{5+3+2}
& \PH{10},  \FF{26},  \FY{30}, \nocomposition \\
$\led{\fr13,\,2\alpha,\,4\alpha,\,4\alpha}$ & 
& \branch{\brep25}{\brep33+1}{4+4+2}
& no covering, \PN{12} \\
$\led{\fr13,\,3\alpha,\,3\alpha,\,4\alpha}$ & 
& \branch{\brep25}{\brep33+1}{4+3+3}
& no covering, \PN{13} \\ 
$\led{\fr12,\,\alpha,\,\alpha,\,7\alpha}$ & 9
& \branch{\brep24+1}{\brep33}{7+1+1}
& \PH{11}, \nocomposition \\
$\led{\fr12,\,\alpha,\,2\alpha,\,6\alpha}$ & 
& \branch{\brep24+1}{\brep33}{6+2+1}
& \PH{12},  \FF{20},  \FY{27},  \comp{\GT3}{3} \\
$\led{\fr12,\,\alpha,\,3\alpha,\,5\alpha}$ & 
& \branch{\brep24+1}{\brep33}{5+3+1}
& \PH{13},  \FF{18},  \FY{26}, \nocomposition\\
$\led{\fr12,\,2\alpha,\,2\alpha,\,5\alpha}$ & 
& \branch{\brep24+1}{\brep33}{5+2+2}
& no covering, \PN{14} \\
$\led{\fr12,\,\alpha,\,4\alpha,\,4\alpha}$ & 
& \branch{\brep24+1}{\brep33}{4+4+1}
& no covering, \PN{15} \\
$\led{\fr12,\,2\alpha,\,3\alpha,\,4\alpha}$ & 
& \branch{\brep24+1}{\brep33}{4+3+2}
& \PH{14},  \FF{25},  \FY{25}, \comp{\GT3}{3}\\
$\led{\fr12,\,3\alpha,\,3\alpha,\,3\alpha}$ & 
& \branch{\brep24+1}{\brep33}{3+3+3}
& no covering, \PN{16} \\ 
$\led{\fr23,\,\alpha,\,\alpha,\,6\alpha}$ & 8
& \branch{\brep24}{\brep32+2}{6+1+1}
& \PH{15},  \FF{14},  \comp{\GT4}{2} \\
$\led{\fr23,\,\alpha,\,2\alpha,\,5\alpha}$ & 
&  \branch{\brep24}{\brep32+2}{5+2+1}
&\PH{16},  \FF{17},  \nocomposition \\
$\led{\fr23,\,\alpha,\,3\alpha,\,4\alpha}$ & 
& \branch{\brep24}{\brep32+2}{4+3+1}
& no covering,  \PN{17}\\
$\led{\fr23,\,2\alpha,\,2\alpha,\,4\alpha}$ & 
& \branch{\brep24}{\brep32+2}{4+2+2}
& no  covering, \PN{18} \\
$\led{\fr23,\,2\alpha,\,3\alpha,\,3\alpha}$ & 
& \branch{\brep24}{\brep32+2}{3+3+2}
& \PH{17}, \FF{13}, \comp{\GT4}{2}\\
$\led{\fr13,\fr13,\,\alpha,\,7\alpha}$ & 
& \branch{\brep24}{\brep32+1+1}{7+1}
& \PH{18}, \nocomposition \\
$\led{\fr13,\fr13,\,2\alpha,\,6\alpha}$ & 
& \branch{\brep24}{\brep32+1+1}{6+2}
& \PH{19},  \FF{16},  \FY{21}, \comp{\GT2}{\BT4}, \comp{\GT4}{2} \\
$\led{\fr13,\fr13,\,3\alpha,\,5\alpha}$ & 
& \branch{\brep24}{\brep32+1+1}{5+3}
& no covering, \PN{19} \\
$\led{\fr13,\fr13,\,4\alpha,\,4\alpha}$ & 
& \branch{\brep24}{\brep32+1+1}{4+4}
& \PH{20},  \FF{15},  \FY{20}, \comp{\GT2}{4},
\comp{\AT4}{\HT2}\\ 
$\led{\fr12,\fr13,\,\alpha,\,6\alpha}$ & 7
& \branch{\brep23+1}{\brep32+1}{6+1}
& \PH{21}, \nocomposition \\
$\led{\fr12,\fr13,\,2\alpha,\,5\alpha}$ & 
& \branch{\brep23+1}{\brep32+1}{5+2} &
\PH{22}, \FF{11}, \FY{18}, \nocomposition \\
$\led{\fr12,\fr13,\,3\alpha,\,4\alpha}$ & 
& \branch{\brep23+1}{\brep32+1}{4+3}
& \PH{23}, \FF{12}, \FY{19}, \nocomposition \\
\hline
\end{tabular}
}
\bigskip
\caption{Possible branching patterns for pull-back transformations
from $\hpgde{1/2,1/3,\alpha}$ to a Heun equation, of degree $D\ge7$.}
\label{tab:clas1} 
\end{center}
\end{table}

\begin{table}
\begin{center}
{
\small
\begin{tabular}{@{}llcll@{}} 
\hline 
\multicolumn{2}{@{}c}{Exponent differences} & Deg. &
Branching pattern & Covering characterization, \\
\cline{1-2} \, hyperg. & \, Heun & $D$ & above singularities & composition \\
\hline
$\led{\fr12,\fr13,\,\alpha} $
& $\led{\fr13,\fr23,\,\alpha,\,5\alpha}$ & 6
& \branch{\brep23}{\brep31+2+1}{5+1}
& \PH{24}, \FF{9}, \nocomposition \\
& $\led{\fr13,\fr23,\,2\alpha,\,4\alpha}$ & 
& \branch{\brep23}{\brep31+2+1}{4+2}
& \PH{25}, \FF{10}, \comp{\GT2}{3} \\
& $\led{\fr13,\fr23,\,3\alpha,\,3\alpha}$ & 
& \branch{\brep23}{\brep31+2+1}{3+3}
& no covering, \PN{20} \\
& $\led{\fr13,\fr13,\fr13,\,6\alpha}$ & 
& \branch{\brep23}{\brep31+1+1+1}{6}
&  \PH{38}, \comp{\GT2}{\CT3} \\
& $\led{\fr12,\fr12,\,\alpha,\,5\alpha}$ & 
& \branch{\brep22+1+1}{\brep32}{5+1}
& \PH{26}, \nocomposition \\
& $\led{\fr12,\fr12,\,2\alpha,\,4\alpha}$ & 
& \branch{\brep22+1+1}{\brep32}{4+2}
&  \PH{27}, \FF{7}, \FY{13}, \comp{\GT3}{2} \\
& $\led{\fr12,\fr12,\,3\alpha,\,3\alpha}$ & 
& \branch{\brep22+1+1}{\brep32}{3+3}
& \PH{28}, \FF{6}, \FY{12}, \comp{\CT3}{\HT2}\\ 
& $\led{\fr12,\fr23,\,\alpha,\,4\alpha}$ & 5
& \branch{\brep22+1}{\brep31+2}{4+1}
&  \PH{29}, \FF{8}, \nocomposition \\
& $\led{\fr12,\fr23,\,2\alpha,\,3\alpha}$ & 
& \branch{\brep22+1}{\brep31+2}{3+2}
&  \PH{30}, \FF{5}, \nocomposition \\
& $\led{\fr12,\fr13,\fr13,\,5\alpha}$ & 
& \branch{\brep22+1}{\brep31+1+1}{5}
&  \PH{37}, \nocomposition \\ 
& $\led{\fr12,\fr12,\fr13,\,4\alpha}$ & 4
& \branch{\brep21+1+1}{\brep31+1}{4} 
& \PH{36}, \nocomposition \\
\hline 
$\led{\fr12,\fr14,\,\alpha}$
& $\led{\alpha,\,\alpha,\,\alpha,\,5\alpha}$ & 8
& \branch{\brep24}{\brep42}{5+1+1+1}
& no covering, \PN{21} \\
& $\led{\alpha,\,\alpha,\,2\alpha,\,4\alpha}$ & 
& \branch{\brep24}{\brep42}{4+2+1+1}
& \PHforty,  \FE{9},  \FY{24}, \compp{\GT2}{\GT2}{2} \\
& $\led{\alpha,\,\alpha,\,3\alpha,\,3\alpha}$ & 
& \branch{\brep24}{\brep42}{3+3+1+1}
& \PH{20},  \FE{8},  \FY{23}, \comp{\GT2}{4}, \comp{\AT4}{\HT2} \\
& $\led{\alpha,\,2\alpha,\,2\alpha,\,3\alpha}$ & 
& \branch{\brep24}{\brep42}{3+2+2+1}
& no covering, \PN{22} \\
& $\led{2\alpha,\,2\alpha,\,2\alpha,\,2\alpha}$ & 
& \branch{\brep24}{\brep42}{2+2+2+2}
& \PHfortyone,  \FE{7}, \FY{22},
$2\!\times\!2\!\times\!2$ \\ 
& $\led{\fr12,\,\alpha,\,\alpha,\,4\alpha}$ & 6
& \branch{\brep23}{\brep41+2}{4+1+1}
& no covering, \PN{23} \\
& $\led{\fr12,\,\alpha,\,2\alpha,\,3\alpha}$ & 
& \branch{\brep23}{\brep41+2}{3+2+1}
& \PH{25},  \FE{6},  \FY{15}, \comp{\GT2}{3} \\
& $\led{\fr12,\,2\alpha,\,2\alpha,\,2\alpha}$ & 
& \branch{\brep23}{\brep41+2}{2+2+2}
& no covering, \PN{24} \\
& $\led{\fr14,\fr14,\,\alpha,\,5\alpha}$ & 
& \branch{\brep23}{\brep41+1+1}{5+1}
& \PHfortytwo, \nocomposition \\
& $\led{\fr14,\fr14,\,2\alpha,\,4\alpha}$ & 
& \branch{\brep23}{\brep41+1+1}{4+2}
& no covering, \PN{23} \\
& $\led{\fr14,\fr14,\,3\alpha,\,3\alpha}$ & 
& \branch{\brep23}{\brep41+1+1}{3+3}
& \PHfortythree,  \FE{14}, \FY{14}, \comp{3}{\HT2} \\ 
& $\led{\fr12,\fr14,\,\alpha,\,4\alpha}$ & 5
& \branch{\brep22+1}{\brep41+1}{4+1}
& \PHfortyfour, \nocomposition \\
& $\led{\fr12,\fr14,\,2\alpha,\,3\alpha}$ & 
& \branch{\brep22+1}{\brep41+1}{3+2}
& \PH{29},  \FE{11}, \FY{10}, \nocomposition \\ 
& $\led{\fr12,\fr12,\,\alpha,\,3\alpha}$ & 4
& \branch{\brep21+1+1}{\brep41}{3+1}
& \PH{36}, \nocomposition \\
& $\led{\fr12,\fr12,\,2\alpha,\,2\alpha}$ & 
& \branch{\brep21+1+1}{\brep41}{2+2} 
& \PH{35},  \FE{10},  \FY{7}, \comp{\GT2}{\HT2} \\
\hline 
$\led{\fr12,\fr15,\,\alpha}$
& $\led{\fr15,\,\alpha,\,\alpha,\,4\alpha}$ & 6
& \branch{\brep23}{\brep51+1}{4+1+1}
& \PHfortytwo, \nocomposition \\
& $\led{\fr15,\,\alpha,\,2\alpha,\,3\alpha}$ & 
& \branch{\brep23}{\brep51+1}{3+2+1}
& \PH{24},  \FE{15},  \FY{16}, \nocomposition \\
& $\led{\fr15,\,2\alpha,\,2\alpha,\,2\alpha}$ & 
& \branch{\brep23}{\brep51+1}{2+2+2}
& no covering, \PN{25} \\ 
& $\led{\fr12,\,\alpha,\,\alpha,\,3\alpha}$ & 5
& \branch{\brep22+1}{\brep51}{3+1+1}
& \PH{37}, \nocomposition \\
& $\led{\fr12,\,\alpha,\,2\alpha,\,2\alpha}$ & 
& \branch{\brep22+1}{\brep51}{2+2+1}
& \PHfortyfive,  \FE{12},  \FY{11}, \nocomposition \\
\hline 
$\led{\fr12,\fr16,\,\alpha}$
& $\led{\alpha,\,\alpha,\,\alpha,\,3\alpha}$ & 6
& \branch{\brep23}{\brep61}{3+1+1+1}
& \PH{38}, \comp{\GT2}{\CT3} \\
&  $\led{\alpha,\,\alpha,\,2\alpha,\,2\alpha}$ & 
& \branch{\brep23}{\brep61}{2+2+1+1}
&  \PHthirtynine,  \FE{13},  \FY{17}, \comp{\GT2}{3}, \comp{3}{\HT2} \\
\hline 
$\led{\fr13,\fr13,\,\alpha}$
&  $\led{\alpha,\,\alpha,\,\alpha,\,3\alpha}$ & 6
& \branch{\brep32}{\brep32}{3+1+1+1}
& no covering, \PN{26} \\
& $\led{\alpha,\,\alpha,\,2\alpha,\,2\alpha}$ & 
& \branch{\brep32}{\brep32}{2+2+1+1}
&\PH{28}, \comp{\GT{\CT3}}{2} \\ 
& $\led{\fr13,\fr13,\,\alpha,\,3\alpha}$& 4
& \branch{\brep31+1}{\brep31+1}{3+1}
&  \PHfortysix, \FT{4},  \FY{9}, \nocomposition \\
& $\led{\fr13,\fr13,\,2\alpha,\,2\alpha}$ & 
& \branch{\brep31+1}{\brep31+1}{2+2}
&  \PHfortyseven, \FT{3},  \FY{8}, \nocomposition \\
\hline 
$\led{\fr13,\fr14,\,\alpha}$     
& $\led{\fr13,\,\alpha,\,\alpha,\,2\alpha}$ & 4
& \branch{\brep31+1}{\brep41}{2+1+1}
& \PH{36}, \nocomposition \\
\hline 
$\led{\fr14,\fr14,\,\alpha}$     
& $\led{\alpha,\,\alpha,\,\alpha,\,\alpha}$ & 4
& \branch{\brep41}{\brep41}{1+1+1+1}
& \PHfortyeight, \comp{2}{\HT2}\\
\hline
\end{tabular}
}
\bigskip
\caption{The other possible branching patterns of Gauss-to-Heun
  transformations with one free parameter.}
\label{tab:clas2} 
\end{center}
\end{table}

\begin{table}
\begin{center}
{
\small
\begin{tabular}{@{}ccll@{}} 
\hline 
Id\; & Deg.  & Branching pattern & A rational expression for $\varphi(x)$\\
\hline
\PH{1\hphantom{2}} &12 & \branch{\brep26}{\brep34}{9+1+1+1} &
$\fr{64\,x^3(x^3-1)^3}{(8x^3-9)}$ \\ 
\PH{2\hphantom{2}} &  & \branch{\brep26}{\brep34}{8+2+1+1} &
$\fr{27\,x^2(x^2-4)}{4\,(x^4-4x^2+1)^3}$ \\ 
\PH{3\hphantom{2}} & & \branch{\brep26}{\brep34}{6+3+2+1} &
$\fr{27\,(x-1)^3(2x-3)^2(x+3)}{4\,x^3(x^3-6x+6)^3}$ \\ 
\PH{4\hphantom{2}} & & \branch{\brep26}{\brep34}{5+5+1+1} &
$\fr{1728\,x^5(x^2-11x-1)}{(x^4-12x^3+14x^2+12x+1)^3}$ \\ 
\PH{5\hphantom{2}} & & \branch{\brep26}{\brep34}{4+4+2+2} & 
$\fr{27\,x^4(x^2-1)^2}{4\,(x^4-x^2+1)^3}$ \\ 
\PH{6\hphantom{2}} & & \branch{\brep26}{\brep34}{3+3+3+3} &
$\fr{-64\,x^3(x^3-1)^3}{(8x^3+1)^3}$ \\ 
\PH{7\hphantom{2}} &10 & \branch{\brep25}{\brep33+1}{8+1+1} & 
$\fr{-4\,(x+2)(x^3+3x+2)^3}{27\,(3x^2-2x+11)}$ \\
\PH{8\hphantom{2}} & & \branch{\brep25}{\brep33+1}{7+2+1} & 
$\fr{4\,(x+4)\,(x^3-6x-2)^3}{27\,(3x+4)^2(4x-11)}$ \\ 
\PH{9\hphantom{2}} & & \branch{\brep25}{\brep33+1}{5+4+1} & $\fr{-(8x-1)(8x^3+87x^2+96x-64)^3}{2^3 3^{12} \,x^4(x+10)}$ \\ %
\PH{10} & & \branch{\brep25}{\brep33+1}{5+3+2} & $\fr{-(x-3)(81x^3-9x^2-53x-27)^3}{2^{14} 3\,x^3(9x+5)^2}$ \\ %
\PH{11} &9 & \branch{\brep24+1}{\brep33}{7+1+1} & 
$\fr{4(x^3+4x^2+10x+6)^3}{27(4x^2+13x+32)}$ \\ %
\PH{12} & & \branch{\brep24+1}{\brep33}{6+2+1} & 
$\fr{27\,x^2(x-3)}{4\,(x^3-3x^2+1)^3}$ \\ 
\PH{13} & & \branch{\brep24+1}{\brep33}{5+3+1} & 
$\fr{-25(5x^3+45x^2+39x-25)^3}{2^{14}3^3\,x^3(3x+25)}$ \\ %
\PH{14} & & \branch{\brep24+1}{\brep33}{4+3+2} & 
$\fr{27\,x^3(3x-4)^2}{4\,(x^3-3x-4)^3}$ \\ 
\PH{15} &8 & \branch{\brep24}{\brep32+2}{6+1+1} & 
$\fr{64\,x^2(x^2-1)^3}{(8x^2-9)}$ \\ 
\PH{16} & & \branch{\brep24}{\brep32+2}{5+2+1} & $\fr{-4\,x^2(x^2+8x+10)^3}{27(2x+1)^2(4x+27)}$ \\ %
\PH{17} & & \branch{\brep24}{\brep32+2}{3+3+2} & 
$\fr{-64\,x^2(x^2-1)^3}{(8x^3+1)^3}$ \\ 
\PH{18} & & \branch{\brep24}{\brep32+1+1}{7+1} & $\fr{-(x^2-13x+49)(x^2-5x+1)^3}{2^6 3^3 x}$ \\ %
\PH{19} & & \branch{\brep24}{\brep32+1+1}{6+2} & 
$\fr{-64\,x^2}{(x^2-1)^3(x^2-9)}$ \\ 
\PH{20} & & \branch{\brep24}{\brep32+1+1}{4+4} & 
$\fr{16\,x^3(2x+1)(x-4)}{(x^2-2x-2)^4}$ \\ 
\PH{21} &7 & \branch{\brep23+1}{\brep32+1}{6+1} &
$\fr{4(x-1)((1+2\omega)x^2-3x-\omega)^3}{(4-(1+3\omega)x)}$ \\
\PH{22} & & \branch{\brep23+1}{\brep32+1}{5+2} & $\fr{x(2x^2-35x+140)^3}{108(14x-125)^2}$ \\ %
\PH{23} & & \branch{\brep23+1}{\brep32+1}{4+3} & $\fr{-(x+27)(16x^2+80x-243)^3}{2^2 3^3 7^7 x^3}$ \\ %
\PH{24} &6 & \branch{\brep23}{\brep31+2+1}{5+1}& $\fr{-x^2(x-3)(x+5)^3}{64(3x+16)}$ \\ %
\PH{25} & & \branch{\brep23}{\brep31+2+1}{4+2} &
 $\fr{-4\,x^3(x-1)^2(x+2)}{(3x-2)^2}$ \\ 
\PH{26} & & \branch{\brep22+1+1}{\brep32}{5+1} & 
$\fr{1728\,x}{(x^2+10x+5)^3}$ \\ 
\PH{27} & & \branch{\brep22+1+1}{\brep32}{4+2} & 
$\fr{27\,x^2}{4\,(x^2-1)^3}$ \\ 
\PH{28} & & \branch{\brep22+1+1}{\brep32}{3+3} & $\fr{(x^2+6x-3)^3}{(x^2-6x-3)^3}$ \\ %
\PH{29} &5& \branch{2+2+1}{3+2}{4+1} & $\fr{4\,x^2(x+5)^3}{27(5x+27)}$ \\ %
\PH{30} & & \branch{2+2+1}{3+2}{3+2} & $\fr{x^2(9x-5)^3}{4(5x-1)^2}$ \\ %
\PH{31} &4& \branch{2+2}{2+2}{2+2} & 
$\fr{-4\,x^2}{(x^2-1)^2}$ \\ 
\PH{32} &2& \branch{2}{2}{1+1} & $x^2$ \\ %
\PH{33} &3& \branch{3}{3}{1+1+1} & $x^3$ \\ %
\PH{34} & &\branch{3}{2+1}{2+1} & $x\,(4x-3)^2$ \\ 
\PH{35} &4&\branch{4}{2+2}{2+1+1} & $4x^2(1-x^2)$ \\ %
\PH{36} & &\branch{4}{3+1}{2+1+1} &  $-x^3\,(3x+4)$ \\ 
\PH{37} &5& \branch{5}{2+2+1}{3+1+1} & $\fr{(x^2+11x+64)(x+3)^3}{2^63^3}$ \\ %
\PH{38} &6& \branch{\brep23}{6}{3+1+1+1} &
$4\,x^3\,(1-x^3\,)$ \\ %
\hline
\PHthirtynine &6& \branch{\brep23}{6}{2+2+1+1} & $x^2\,(4x^2-3)^2$ \\ 
\PHforty &8& \branch{\brep24}{\brep42}{4+2+1+1} &
$\fr{4x^2(x^2-2)}{(x^2-1)^4}$ \\ 
\PHfortyone & & \branch{\brep24}{\brep42}{2+2+2+2}  &
$\fr{-4\,x^4}{(x^4-1)^2}$ \\ 
\PHfortytwo &6& \branch{\brep23}{4+1+1}{5+1} & $\fr{-(x-1)^4(x^2-6
x+25)}{256\,x}$ \\ %
\PHfortythree & & \branch{\brep23}{4+1+1}{3+3} &
$\fr{27\,x\,(x-4)}{4\,(x^2-4x+1)^3}$ \\  
\PHfortyfour &5& \branch{4+1}{4+1}{2+2+1} &
$\fr{x\,(x-1-2i)^4}{((1+2i)x-1)^4}$ \\ %
\PHfortyfive & & \branch{5}{2+2+1}{2+2+1} & $\fr{x\,(x^2-5x+5)^2}4$ \\ %
\PHfortysix &4& \branch{3+1}{3+1}{3+1} & $\fr{x^3\,(x+2)}{(2x+1)}$ \\ %
\PHfortyseven & & \branch{3+1}{3+1}{2+2} &
$\fr{64\,x\,(x-1)^3}{(8x-9)}$ \\ 
\PHfortyeight & & \branch{4}{4}{1+1+1+1} & $x^4$ \\ %
\hline
\end{tabular}
}
\bigskip
\caption{The Belyi coverings appearing in Gauss-to-Heun pull-backs, 
  up~to M\"obius transformations. }
\label{tab:coverings}
\end{center}
\end{table}

We proceed to explain the results and notation in Tables \ref{tab:clas}--\ref{tab:coverings}.
Let $\omega$ denote a primitive cubic root of unity,
say $\omega=\exp(2\pi i/3)$. In particular, $\omega^2+\omega+1=0$. 

The pull-back transformations from a hypergeometric equation $E_1$
to a Heun equation $E_2$ are classified and demonstrated in the following four steps.
They parallel the principal steps \refpart{i}--\refpart{iii} outlined in the introduction,
with the only difference that Step \refpart{i} is split into two steps.

{\bf Step 1} is determination of possible restrictions on the exponent differences 
of $E_1$ and the degree of the pull-backs. This step is elaborated in \S\ref{subsec:step2}.
The restrictions on the exponent differences determine the {\em type} of possible 
branching patterns, which is by~definition an unordered list of the integers $k\ge 1$
that determine the restricted exponent differences $1/k$.
The following list of types is obtained:
\begin{equation}
(), \ (2), \  (3), \ (2,3), \ (2,4), \ (2,5), \ (2,6), \ (3,3), \ (3,4), \ (4,4).
\end{equation}
The first type $()$ means no restrictions on the parameters of $E_1$. 
We skipped the types $(1)$ and $(2,2)$; they are considered in \cite{VidunasHDD} as mentioned.
The types are indicated by the exponent differences of $E_1$ 
in the first columns of Tables \ref{tab:clas}, \ref{tab:clas2}, and the whole Table~\ref{tab:clas1}
is devoted to the type $(2,3)$. The entries of different types are separated by horizontal lines.
The pull-back degree is given in the third columns of 
Tables \ref{tab:clas}, \ref{tab:clas2}, and the second column of Table~\ref{tab:clas1}. 
The maximal degree is 12. It occurs for the type $(2,3)$ only.

{\bf Step 2} is determination of possible branching patterns. The method is explained in 
\S \ref{subsec:step3}. The result is presented by the fourth columns of
Tables \ref{tab:clas}, \ref{tab:clas2}, and the third column of Table \ref{tab:clas1}.
Generally, we indicate a branching pattern by an (unordered) list of three
unordered partitions of its degree~$D$, separated by the equality signs.
The partitions specify the branching indices in each of the three branching
fibers of a Belyi covering. Besides, we use the abbreviation \brep{k}{n}~ 
for a partition block $k+\dots+k$ ($n$~times). 
In Tables \ref{tab:clas}, \ref{tab:clas1}, \ref{tab:clas2}, 
the symbol \brep{k}{n} specifically means presence of $n$~points of~$\EQN_2$ 
with the branching order $k$ above a singular point of $\EQN_1$ 
with the the exponent difference $1/k$.
By part \refpart{c} of Lemma \ref{lem:genrami}, each of the $n$ points will be either 
ordinary or an irrelevant singularity for $\EQN_2$. By the described convention,
the branching patterns for pull-back transformations with $M$ free parameters have 
$3-M$ numbers (i.e., branching orders) inclosed in square brackets, and exactly $4$ 
non-bracketed numbers representing the $4$~singular points of $\EQN_2$. 

In total,  we get a list of 89 branching patterns, though some of the patterns differ only
by the square-brackets specification of ordinary points of $E_2$. For example, two degree 
3 branching patterns in Table \ref{tab:clas} are the same, leading to the same 
cubic covering $H_{34}$  (identified in the last column).
The exponent differences of $E_2$ are determined by $E_1$ and the branching pattern,
and are given by the second columns of Tables \ref{tab:clas}, \ref{tab:clas2} and 
the first column of Table \ref{tab:clas1}.

{\bf Step 3} is computation of the Belyi coverings $\varphi\colon\PP^1_x\to\PP^1_z$. 
Generally, computation of Belyi maps with a given branching pattern is a difficult problem. 
However, the maximal degree implied by the possible branching patterns is just 12. 
With the aid of modern computer algebra systems 
this problem is tractable for coverings of degree $12$ or less, 
even using a straightforward Ansatz method with undetermined coefficients.
Most of the Belyi maps are actually known in the literature, if only because
the Belyi maps of the type (2,3) occur in Herfurtner's list~\cite{Herfurtner91}
of elliptic surfaces with four singular fibers. Specifically, the 
$\mathcal{J}(X,Y)$-expressions in \cite[Table~3]{Herfurtner91} are homogeneous 
expressions of  the Belyi maps $H_1,\dots,H_{38}$ up to M\"obius transformations.
Moreover, the same coverings basically appear in pull-backs between 
hypergeometric equations, because a free parameter can always be specialized so
to reduce the Heun equation $E_2$ to a hypergeometric (or simpler) equation.

The computational issues of Step 3 are discussed in \S \ref{subsec:step3}. 
Complementarily, \S \ref{sec:6} presents an elegant approach
to show non-existence of Belyi maps with many branching patterns.
The full list of computed Belyi maps is given in Table~\ref{tab:coverings},
and further commented in \S \ref{sec:5}.
The last columns of Tables \ref{tab:clas}, \ref{tab:clas1}, \ref{tab:clas2}
identify the Belyi map for each possible pull-back transformation.
These columns also specify the {\em Coxeter decompositions} \cite{Felikson98} and
{\em divisible tilings} \cite{Broughton2000} for the Schwarz maps associated to the pulled-back 
Heun's equation $E_2$ (by various $F$-numbers, as explained in \S \ref{subsec:Ftilings}),
and describe 
composite transformations by product expressions indicating degrees
of occurring  indecomposable transformations. 
The product notation has to be followed
from right to left to trace the composition from the starting hypergeometric equation.
The factor $2_H$ denotes quadratic Heun-to-Heun transformation 
(\ref{eq:quadtr3}). Here is the meaning of other indexed degrees: 
$\CT3$ denotes the cyclic covering \PH{33} with the branching pattern \branch{3}{3}{1+1+1},
while $\AT4$ and $\BT4$ stand for the coverings \PH{36} (\branch{4}{3+1}{2+1+1})
and \PH{46} (\branch{3+1}{3+1}{3+1}), respectively.
The unindexed numbers 3 and 4 denote the frequent coverings \PH{34} (\branch{3}{2+1}{2+1})
and \PH{47} (\branch{3+1}{3+1}{2+2}), respectively. 
In any composition, there is exactly one factor representing an indecomposable Gauss-to-Heun
transformation; it is the first one from the left which is not~$2_H$. The other factors
to the right represent pull-backs between hypergeometric equations.
The notation $\DDT$ indicates a composition of quadratic transformations that can be realized
in multiple ways, possibly including $\HT2$; see (\ref{eq:compDD}) below for the most typical
example. The compositions are considered more thoroughly in \S \ref{subsec:composite} 
and in \cite[Appendix B]{Vidunas2009b}.

There are 27 different branching patterns for which there is no Belyi map.
The non-existence in all these cases can be elegantly shown by considering
implied (but not possible) pull-back transformations between Fuchsian equations,
as explained in \S \ref{sec:6}. The indexed $N$-notation refers to Table \ref{tab:nonexist} below.
For each branching pattern except two leading to $\PH{21}$ and $\PH{44}$, there
is at most one covering (and one pull-back) up to M\"obius transformations. 
The coverings $\PH{21}$ and $\PH{44}$
are defined, respectively, over $\QQ(\omega)$ and $\QQ(i)$. 
In either of these cases, we actually have a complex-conjugated pair of Belyi coverings.
Table~\ref{tab:coverings} lists 48 different coverings, though $\PH{21}$ and $\PH{44}$
should be properly counted twice. It is instructive to compare the branching pattern and
the orders of vertices and cells of the {\em dessins d'enfant} in Figure \ref{fg:dessins}.
In total, we count 61 parametric pull-backs 
among the entries of Tables \ref{tab:clas}, \ref{tab:clas1}, \ref{tab:clas2}.
Of them, 28 are composite. Evidently, some of the 48 coverings appear in more than one 
pull-back.  Accordingly, the symbol~\brep{k}{n} in Table~\ref{tab:coverings} 
merely indicates  presence of $n$~points of branching order~$k$ in the same fiber.
The coverings \PH{20}, \PH{24}, \PH{25}, \PH{28}, \PH{29},
\PH{34}, \PH{35}, \PH{37}, \PH{38}, \PH{42}, \PH{47} appear twice in
Tables \ref{tab:clas}, \ref{tab:clas1}, \ref{tab:clas2}, while \PH{36} three times.

{\bf Step 4} is derivation of identities between standard $\hpgo{2}{1}(z)$ and $\Hn(x)$ solutions
of the related hypergeometric and Heun equations, with $z=\varphi(x)$. This gives 
{\em Heun-to-hypergeometric reduction} formulas, expressing found Heun functions in
terms of the better understood Gauss hypergeometric functions. 
This final step is comprehensively considered in the parallel paper \cite{Vidunas2009b}
by the same authors. In particular, \cite[\S 3]{Vidunas2009b} explains
the technical issue of choosing the gauge prefactor $\theta(x)$ in 
pull-back transformations (\ref{eq:algtransf}). The transformations without a prefactor
(i.e., $\theta(x)=1$) are classified by Maier in \cite{Maier03}. The branching patterns 
for these pull-backs typically have a fiber with just one point, 
and that point is a singularity  for $E_2$. 
There are 7 of these pull-back transformations. Their type is $()$, $(2)$, $(3)$ or $(2,3)$,
and the coverings are numbered consequently from $H_{32}$ to $H_{38}$. 
Formulas without a prefactor arise from the transformations of Tables \ref{tab:clas}, \ref{tab:clas1} 
realized by these coverings, 
except for the type $(3)$ transformation  with the covering $H_{34}$.
The well-known quadratic transformation (\ref{eq:quadtr2})
is described at the beginning of this article.

Hereby we complete the description of four classification steps. At the same time,
we explained the results and notation in Tables  \ref{tab:clas}--\ref{tab:coverings}.
The next two subsections give a methodological proof of Steps 1 and 2.
Section \ref{sec:5} discusses computational issues of Step 3, composite coverings,
and relations of the recorded transformations to the Herfurtner's list \cite{Herfurtner91}
of elliptic surfaces and Felikson's list \cite{Felikson98} of Coxeter decompositions. 
Section \ref{sec:6} describes the elegant approach of proving non-existence
of Belyi coverings with certain branching patterns, and applies it to Tables 
\ref{tab:clas}--\ref{tab:clas2} and the Miranda-Persson list \cite{Miranda89}
of degree 24 branching patterns.

\subsection{Step 1: Possible restricted exponent differences and degree}
\label{subsec:step2}

We are looking for the Belyi coverings $\varphi\colon\PP_x^1\to\PP_z^1$
that pull-back a hypergeometric equation $E_1$ to Heun's equation $E_2$. 
We assume that $E_1$ 
is not specifically of the form $E(1,\alpha,\beta)$ or $E(1/2,1/2,\alpha)$,
because then either it has a logarithmic singularity 
(if~$\beta\neq\pm\alpha$ by \cite[Lemma 5.1]{Vidunas2009}; it
would not  contribute extra non-singular points above $\{0,1,\infty\}\subset\PP_z^1$ 
necessary for new cases of Gauss-to-Heun pull-backs), 
or it has cyclic (if $\beta=|\alpha|$)
or dihedral monodromy as explored in \cite{VidunasHDD}.

We restrict $m\in\{0,1,2\}$ exponent differences of the general hypergeometric 
equation (\ref{eq:GHE}) to the reciprocals of integers $k>1$,
and look for particular cases when part {\em (c)} of Lemma \ref{lem:genrami}
allows enough non-singular points above $\{0,1,\infty\}\subset\PP_z^1$.
The degree of $\varphi$ is denoted by $D$.

First, assume that $m=0$. This puts no restriction on the exponent differences of~$\EQN_1$, 
so generally 
all points above $z=0,\,1,\,\infty$ will be singularities of~$\EQN_2$. 
There will be exactly $D+2$ singular points by
Lemma~\ref{lem:genrami2}, and we wish the transformed equation to be Heun's.
Hence $D+2\leq 4$, so that $D\leq 2$.  If $D=1$ then $\varphi$~is 
a M\"obius transformation and does not 
alter the number of singular points, hence $\EQN_2$~will have only three.  If $D=2$, then
$\varphi$~is a quadratic covering with the branching pattern ${\branch{2}{2}{1+1}}$. 
The pull-back is then the well-known quadratic transformation (\ref{eq:quadtr2}),
applicable to any hypergeometric equation.
We do not need to consider quadratic transformations ($D=2$) subsequently,
nor the case $m=0$ in more detail in the other steps.

For $m\in\{1,2\}$, the number of non-singular points above the restricted singularities 
of $E_1$ must be at least $(D+2)-4=D-2$.

If $m=1$, we allow two free parameters.
We restrict just one exponent difference of~$\EQN_1$ to
equal $1/k$, with integer $k>1$.  The pulled-back equation~$\EQN_2$ will have
at~most $\lfloor D/k\rfloor$ ordinary points above
$\{0,\,1,\,\infty\}\subset\PP_z^1$ by Lemma~\ref{lem:genrami}, and one must
have
\begin{equation}
\left\lfloor \frac{D}k \right\rfloor\ge D-2.
\end{equation}
This leads to the Diophantine inequality
\begin{equation}
\frac2{D}+\frac1k\geq 1.
\end{equation}
For $k>1$ and $D>2$, we have the following possibilities:
\begin{equation}
(k,D)\in \{ (2,3),\;  (2,4),\; (3,3) \}.
\end{equation}
The resulting branching patterns are of types $(2)$, $(3)$, according to~$k$.

If $m=2$, we allow one free parameter.
Suppose that the restricted exponent differences of~$\EQN_1$ equal $1/k$, $1/\ell$, 
where $k,\ell$ are integers.  We assume \mbox{$1<k\le\ell\le D$}
without loss of generality; the last inequality allows actual utilization of the restriction $1/\ell$.
The transformed equation has at~most
$\lfloor D/k\rfloor+\lfloor D/\ell\rfloor$ ordinary points above
$\{0,\,1,\,\infty\}\subset\PP_z^1$.  Similarly to the above, one must have
\begin{equation}
\left\lfloor \frac{D}k\right\rfloor+\left\lfloor\frac{D}{\ell}\right\rfloor\ge D-2,
\end{equation}
which leads to the weaker Diophantine inequality
\begin{equation}
\frac2{D}+\frac1k+\frac1\ell \geq 1.
\end{equation}
Discarding $k=\ell=2$, we get the following possibilities for
$k,\ell$ and for the upper bound~$D_{\max}$ on the degree~$D$:
\begin{equation}
\label{eq:ledt}
(k,\ell,D_{\max})\in\{
(2,3,12),\; (2,4,8),\; (2,5,6),\;  (2,6,6),\;
(3,3,6),\; (3,4,4), \; (4,4,4) \}.
\end{equation}
The resulting branching patterns are the seven types $(2,3)$,
$(2,4)$, $(2,5)$, $(2,6)$, $(3,3)$, $(3,4)$, $(4,4)$, respectively.

\subsection{Step 2: Possible branching patterns}
\label{subsec:step3}

Here we look at each possible type and degree, and determine all branching patterns fitting them.
The constraint that there must be at~least $D-2$ ordinary points above~$\{0,1,\infty\}\subset\PP_z^1$ 
will usually require taking the number of ordinary points above the points with
restricted exponent differences is \emph{maximal}, i.e., equal to
$\left\lfloor D/k \right\rfloor$ or $\left\lfloor D/\ell \right\rfloor$.

The \emph{a priori} possible branching patterns for the case $m=1$
are straightforward to determine. They are listed in the
fourth column of Table~\ref{tab:clas}.  That table 
is comparable to \cite[Table~1b]{Movasati2009}.

In the case $m=2$, we start with the coverings of the type $(2,3)$
of the maximal degree $D=12$, 
as in Table~\ref{tab:clas1}. There must be $12-2=10$ ordinary points
above the two singular points of $\hpgde{1/2,1/3,\alpha}$ with exponent
differences $1/2$ or~$1/3$; all $x$-points in these two fibers must be
ordinary, as $\left\lfloor12/2\right\rfloor+\left\lfloor12/3\right\rfloor=10$.  
The third fiber is a partition of 12 with 4 parts. There are $15$ such partitions,
and they are all listed in the third column of Table~\ref{tab:clas1}.
Next, there are no transformations of degree~$11$, because
$11-2>\left\lfloor11/2\right\rfloor+\left\lfloor11/3\right\rfloor$ and
there would not be enough ordinary points in the two fibers.  In a similar
way, the pull-back coverings of degree $D=10$, $9$, $8$ or~$7$ must have
the maximal number of ordinary points in the two restricted fibers; and all
branching patterns consistent with this constraint are listed.
The branching patterns of type $(2,3)$ continue in Table~\ref{tab:clas2}.
The degrees $D=6$, $4$ require less than 
$\left\lfloor D/2\right\rfloor+\left\lfloor D/3\right\rfloor$ ordinary points
in the restricted fibers, and there is some choice of how to split 
a bracketed number $[2]$ or $[3]$ into a pair of non-bracketed numbers, 
though at least one bracketed number must remain in the two restricted fibers. 
For $D=5$, there is a choice of splitting (or not splitting) the number $2$ 
in the $[3]$ fiber. In total, we get 53 branching patterns of the type $(2,3)$,
all different. 

The other types $(2,4)$, $(2,5)$, $(2,6)$, $(3,3)$, $(3,4)$, $(4,4)$ similarly
give less numerous sets of branching patterns, some of them coinciding mutually
or with previously encountered ones.

\section{The Belyi coverings}
\label{sec:5}

First, this section briefly explains computation of Belyi maps and utilizing specialization
of parametric Gauss-to-Heun transformations to 
transformations between  hypergeometric equations. 
In \S\S \ref{subsec:Herfurtner}--\ref{subsec:Ftilings},
we explain how the $H$-numbering of Table~\ref{tab:coverings} comes partly 
from an algebraic-geometric classification of Herfurtner~\cite{Herfurtner91},
and clarify the various $F$-numbers in the last columns of Tables \ref{tab:clas}--\ref{tab:clas2}
as representing {\em Coxeter decompositions} of Felikson~\cite{Felikson98}
and {\em divisible tilings} of \cite{Broughton2000}. Lastly, in \S \ref{subsec:composite}
we examine the composite transformations among our results.

\subsection{Computational issues}
\label{subsec:step4}

To compute the Belyi maps $\varphi\colon\PP^1_x\to\PP^1_z$ with a given branching pattern
means to find all rational functions $\varphi(x)$ such that the numerators of $\varphi(x)$,
$1-\varphi(x)$ and the denominator of $\varphi(x)$ factor according to the branching pattern. 
A straightforward Ansatz method with undetermined coefficients can be used for low degree
coverings. Modern computer algebra systems (such as {\sf Maple} and {\sf Mathematica})
can handle the resulting systems of algebraic equations easily if the degree of $\varphi(x)$ 
is 12 or less. More cannily, one can consider factorization of the numerators of the logarithmic
derivatives of $\varphi(x)$ and $\varphi(x)-1$ as in \mbox{\cite[\S\,3]{Vidunas2005}}. 
For example, to determine $H_1$, one is looking for a constant $c$ and
monic polynomials $P,Q,R$ of degree 4, 3, 6, respectively, such that $\varphi(x)=c\,P^3/Q$ and
\mbox{$\varphi(x)-1=c\,R^2/Q$}. To find these polynomials, one considers
\begin{equation}
\frac{\varphi'(x)}{\varphi(x)}=\frac{3P'}{P}-\frac{Q'}{Q} \stackrel{!}=\frac{9R}{P\,Q},\qquad
\frac{(\varphi(x)-1)'}{\varphi(x)-1}=\frac{2R'}{R}-\frac{Q'}{Q} \stackrel{!}=\frac{9P^2}{R\,Q}.
\end{equation}
Zeroes of the derivatives are the branching points other than in the denominators,
and the factor 9 is determined by local consideration at $x=\infty$.  
The whole polynomial $R$ can be eliminated symbolically using the first identification,
and the resulting equation system for the undetermined coefficients of $P$, $Q$ is rather transparent.
In general, a covering with a given branching pattern may
not exist, or there may be several Belyi maps (up to M\"obius equivalence)
or even several $\overline{\QQ}/\QQ$-Galois orbits of Belyi maps with the same branching pattern.
The Galois action on the Belyi maps and their {\em dessins d\'enfant} is of primary interest
to Grothendieck's theory \cite{Schneps94}, \cite{Shabat2000}.

Less demandingly, one may notice that the free parameter of our Gauss-to-Heun transformations
can be specialized so that to the pulled-back Heun equation has actually less than 4 singular points.
Therefore, the Belyi coverings must appear in the classification \cite{Vidunas2009} of 
Gauss-to-Gauss transformations 
in principle,  though there are a few infinite families of those transformations 
(for degenerate, dihedral, algebraic or elliptic hypergeometric functions). 
In particular, each branching pattern of Table~\ref{tab:clas} 
can be found in \mbox{\cite[Table~1]{Vidunas2009}}, except for \branch{2+2}{2+2}{2+2} which
corresponds to the transformation $\hpgde{1/2,1/2,\alpha}\pback{4}\hpgde{1,2\alpha,2\alpha}$ 
briefly mentioned in \cite[p.~161]{Vidunas2009}.
The branching patterns of Tables \ref{tab:clas1} and~\ref{tab:clas2} (with
$m=1$ free parameter) can be handled similarly, yielding reductions 
of one-parameter Gauss-to-Heun transformations to 
zero-parameter pull-backs between hypergeometric functions.
For example, the covering \PH{27} implies the hypergeometric transformations
$\hpgde{1/2,1/3,1/2}\pback{6}\hpgde{1/2,1/2,2}$
and $\hpgde{1/2,1/3,1/4}\pback{6}\hpgde{1/2,1/2,1/2}$.
These specializations reductions are possible whenever there is a branching point with a free
exponent difference.  Among~the relevant branching patterns, 
only the last one 
(\branch{4}{4}{1+1+1+1}) in Table~\ref{tab:clas2} does not satisfy this
condition. But even~it represents a nominally hypergeometric
transformation, namely $\hpgde{\alpha,\alpha,1}\pback{4}\hpgde{4\alpha,4\alpha,1}$.
Section \ref{sec:6} gives more details for obtaining the list of Gauss-to-Heun pull-backs
from the classification in \cite{Vidunas2009}. In particular, the non-unique coverings 
$\PH{21}$ and $\PHfortyfour$ come from Lemma \ref{lem:elliptic}.

\subsection{The Herfurtner classification}
\label{subsec:Herfurtner}

Pull-back transformations from hypergeometric equations of the form
$\hpgde{1/2,1/3,\alpha}$ to Heun equations have a close relation to
elliptic surfaces over $\CC(x)$ with $4$ singular
fibers~\cite{Herfurtner91,Movasati2009}.  The Belyi coverings
$z=\varphi(x)$ that induce these transformations appear as $j$-invariants
of the elliptic surfaces, with $z$ equal to $\mathcal{J}\defeq j/1728$, the
traditional Klein $j$-invariant.

The elliptic surfaces with $4$~singular fibers are classified by
Herfurtner~\cite{Herfurtner91}.  His article lists $50$~configurations of
singular fibers which give such elliptic surfaces, and for each
configuration, supplies a formula $\mathcal{J}=\mathcal{J}(X,Y)$ which is a
projectivized version of $z=\varphi(x)$, up~to a M\"obius transformation
of~$x$ and a permutation of $z=0,1,\infty$.  Heun equations arise from $38$
of his $50$~cases, as Movasati and Reiter~\cite{Movasati2009} recently
observed.  We adopt the enumeration of~\cite[Table~1]{Movasati2009}, and
denote these $38$~Belyi coverings of Herfurtner, which were not
originally numbered, by \PH{1} to \PH{38}.  
The ordering is by degree in two ranges, as evident in Table~\ref{tab:coverings}:
decreasing in the range $H_{1},\dots,H_{31}$, and increasing in the range
$H_{32},\dots,H_{38}$.

By examining Table \ref{tab:clas1} and the upper part of Table \ref{tab:clas2}, 
one finds that the 
coverings $H_1,\dots,H_{30}$ and $H_{36},\dots,H_{38}$ induce 
Gauss-to-Heun pull-backs of the type $(2,3)$ with one free
parameter. These transformations use each of these $34$
coverings exactly once, and no~other coverings appear.  
The ordering by decreasing degree make the $H$-numbers appear
ordered in Table \ref{tab:clas1}, and almost ordered in
the upper part of Table \ref{tab:clas2}.
By examining Table~\ref{tab:clas},  one finds Herfurtner's coverings $H_{31},\dots,H_{35}$
(with $H_{34}$ appearing twice) and a ``new" covering $H_{47}$. 
The covering $H_{47}$ cannot pull-back $\hpgde{1/2,1/3,\alpha}$ to a Fuchsian equation 
with exactly 4 singularities. The pattern \branch{\brep31}{2+1}{2+1} for $\PH{34}$
cannot be refined to such a pull-back from $\hpgde{1/2,1/3,\alpha}$ either,
but this is possible for the other $\PH{34}$ parsing \branch{\brep21+1}{2+1}{3}.
This explains why $\PH{34}$ appears in Herfurtner's list once.

Some of Herfurtner's coverings additionally induce one-parameter Gauss-to-Heun
transformations of types $(2,4)$, $(2,5)$, etc., as evident in
Table~\ref{tab:clas2}.  But $10$~extra coverings appear in the latter sections of that table;
they have no interpretation in~terms of elliptic surfaces.  
We denote them $\PH{39},\dots,\PH{48}$, ordered somewhat arbitrarily
in the lower part of Table~\ref{tab:coverings}.
The covering~$\PHfortyseven$ induces transformations of the types $(2)$ and $(3,3)$.

\subsection{Coxeter decompositions}
\label{subsec:Ftilings}

Recall that a \emph{Schwartz map} for an second order differential equation in
the complex domain is a map $\CC\to\CC$ defined as the ratio of a pair of
independent solutions of the differential equation~\cite{Beukers2007}.
Consider a hypergeometric equation with real exponent differences $(\alpha,\beta,\gamma)$
satisfying $0\le \alpha,\beta,\gamma<1$.  The image of the upper half plane
under its Schwarz map is a curvilinear \emph{Schwarz triangle}; the sides
are line or circle segments, and the angles are equal to
$\pi\alpha,\pi\beta,\pi\gamma$.  Similarly, consider a Heun equation with real exponent
differences $(\alpha,\beta,\gamma,\delta)$ satisfying $0\le
\alpha,\beta,\gamma,\delta<1$.  The image of the upper half plane under its
Schwarz map is a curvilinear \emph{Schwarz quadrangle}, with the same
kind of sides, and angles are equal to $\pi\alpha,\pi\beta,\pi\gamma,\pi\delta$.

It was noticed by Hodgkinson \cite{Hodgkinson18,Hodgkinson20} that if the covering $\varphi(x)$
of a pull-back transformation between hypergeometric equations is defined over~$\RR$,
the analytic continuations of their solutions according to the Schwarz reflection principle are
compatible. In consequence, the covering~$\varphi$ (of degree~$D$, say)
will induce a subdivision of a Schwarz triangle of the pulled-back hypergeometric equation
into $D$~Schwarz triangles of the original hypergeometric equation. 
Examples of such subdivisions are given in \cite[Figure 1]{Vidunas2005}.

Similarly, suppose we have a Gauss-to-Heun 
transformation defined over $\RR$. In particular, the fourth singular point $x=t$ is real. 
Then the analytic continuations of the hypergeometric and Heun solutions
according to the Schwarz reflection principle are compatible, and
the covering~$\varphi$ (of degree~$D$) will induce a subdivision of a 
Schwarz quadrangle of the Heun equation into
$D$~Schwarz triangles of the hypergeometric equation.

In the context of hyperbolic geometry, the possible subdivisions of
curvilinear quadrangles (or triangles) into curvilinear triangles have been classified by
Felikson~\cite{Felikson98}; they are called \emph{Coxeter decompositions}.
The triangles have angles $\pi\alpha,\pi\beta,\pi\gamma$ satisfying
$\alpha+\beta+\gamma<1$.  The Coxeter decompositions with 
a free (angle) parameter are depicted in Figures 10, 11, 14 in~\cite{Felikson98}.
The subdivisions of Schwarz quadrangles into Schwarz triangles
induced by our Gauss-to-Heun transformations defined over~$\RR$
have the same shape. In Tables \ref{tab:clas}--\ref{tab:clas2}, 
\begin{itemize}
\item the notation \FF{k} refers to the $k$th subdivision picture in \cite[Figure 14]{Felikson98};
these subdivisions are applicable to Gauss-to-Heun pull-backs of the type $(2,3)$;
\item \FE{k} similarly refers to \cite[Figure 11]{Felikson98};
these subdivisions are applicable the pull-backs of the types $(2)$, $(2,4)$, $(2,5)$, $(2,6)$;
\item \FT{k} similarly refers to \cite[Figure 10]{Felikson98};
these subdivisions are applicable the pull-backs of the type $(3)$ or $(3,3)$.
\end{itemize}

Figure \ref{fg:felixon}\refpart{a} depicts the Coxeter dexomposition \FE{13} of a quadrangle with the
angles $\pi\alpha,\pi\alpha,2\pi\alpha,2\pi\alpha$ 
into 6 hyperbolic triangles with the angles $\pi/2,\pi/6,\pi\alpha$. 
It gives a decomposition of a Schwarz quadrangle for 
$\heunde{\alpha,\alpha,2\alpha,2\alpha}$ into Schwarz triangles for $\hpgde{1/2,1/6,\alpha}$
induced by the type (2,6) transformation with the covering $H_{39}$. 
The Schwarz reflection principle is applied to a few edges intersecting at a common vertex.
The decompositions $3\cdot 2$ and $2_H\cdot 3$ are clearly visible in the Coxeter decomposition.
Consequently, the picture also illustrates the decomposition \FT{2} of the same quadrangle into 3 triangles 
with the angles $\pi/3,\pi\alpha,\pi\alpha$, and the decomposition \FF{2} of a quadrangle with the angles
$\pi/2,\pi/2,\alpha,2\alpha$.  Both decompositions are induced by the cubic covering $H_{34}$.
The factor $2_H$ represents a Schwarz reflection between two smaller quadrangles.

Figure \ref{fg:felixon}\refpart{b} is not a quadrangle, of course.  But it contains two Coxeter 
decompositions for Gauss-to-Heun transformations of the type $(3,3)$. 
If we remove the upper black triangle, we get the decomposition \FT{3} of  a quadrangle
with the angles $\pi/3,\pi/3,2\pi\alpha,2\pi\alpha$. If the left white triangle is removed, 
the decomposition \FT{4} of  a quadrangle with the angles $\pi/3,\pi/3,\pi\alpha,3\pi\alpha$
is obtained. The coverings are $H_{47}$ and $H_{46}$, respectively.

\begin{figure}
\[ \begin{picture}(340,285)
\put(141,88){\includegraphics[height=214pt]{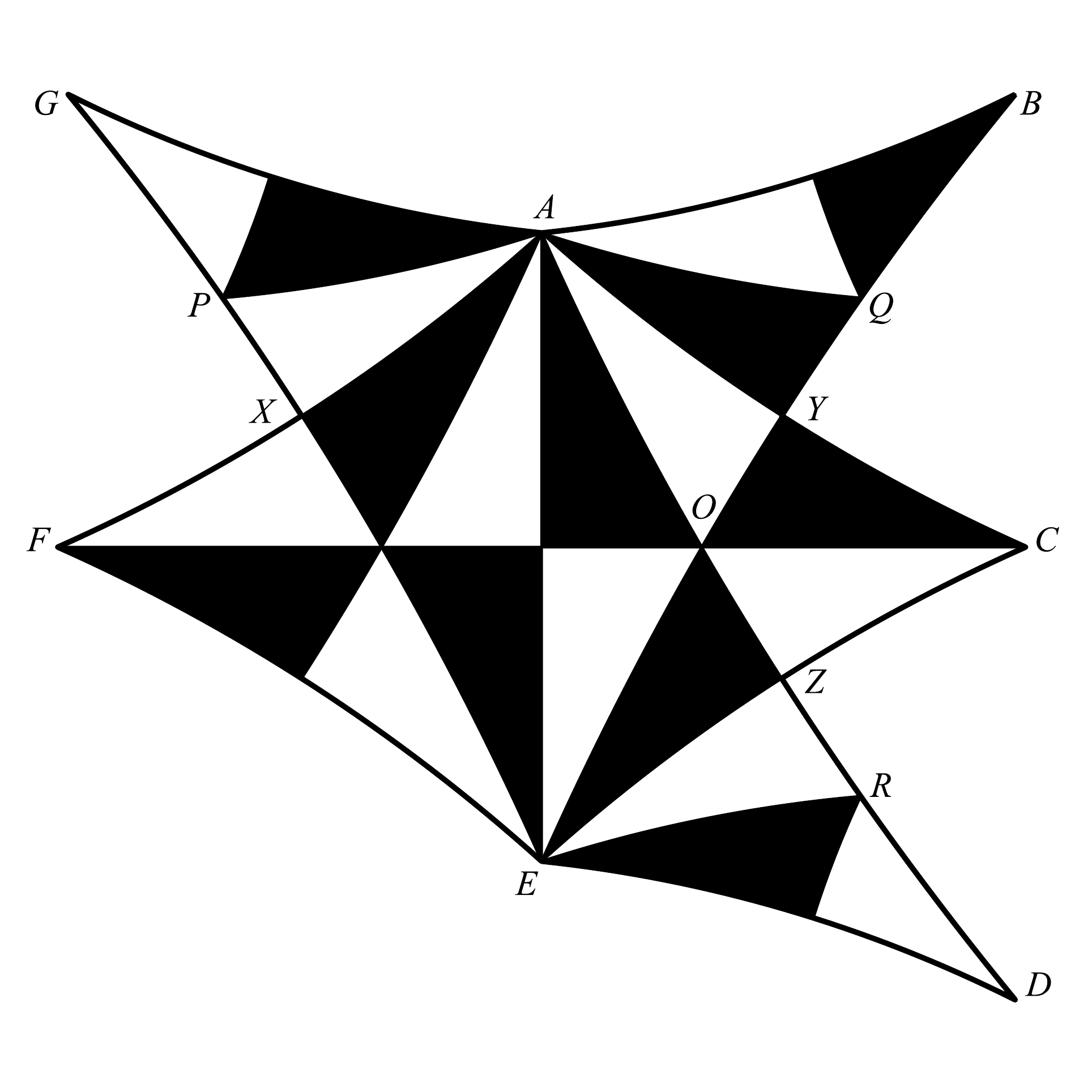}}
\put(-12,-46.5){\includegraphics[height=214pt]{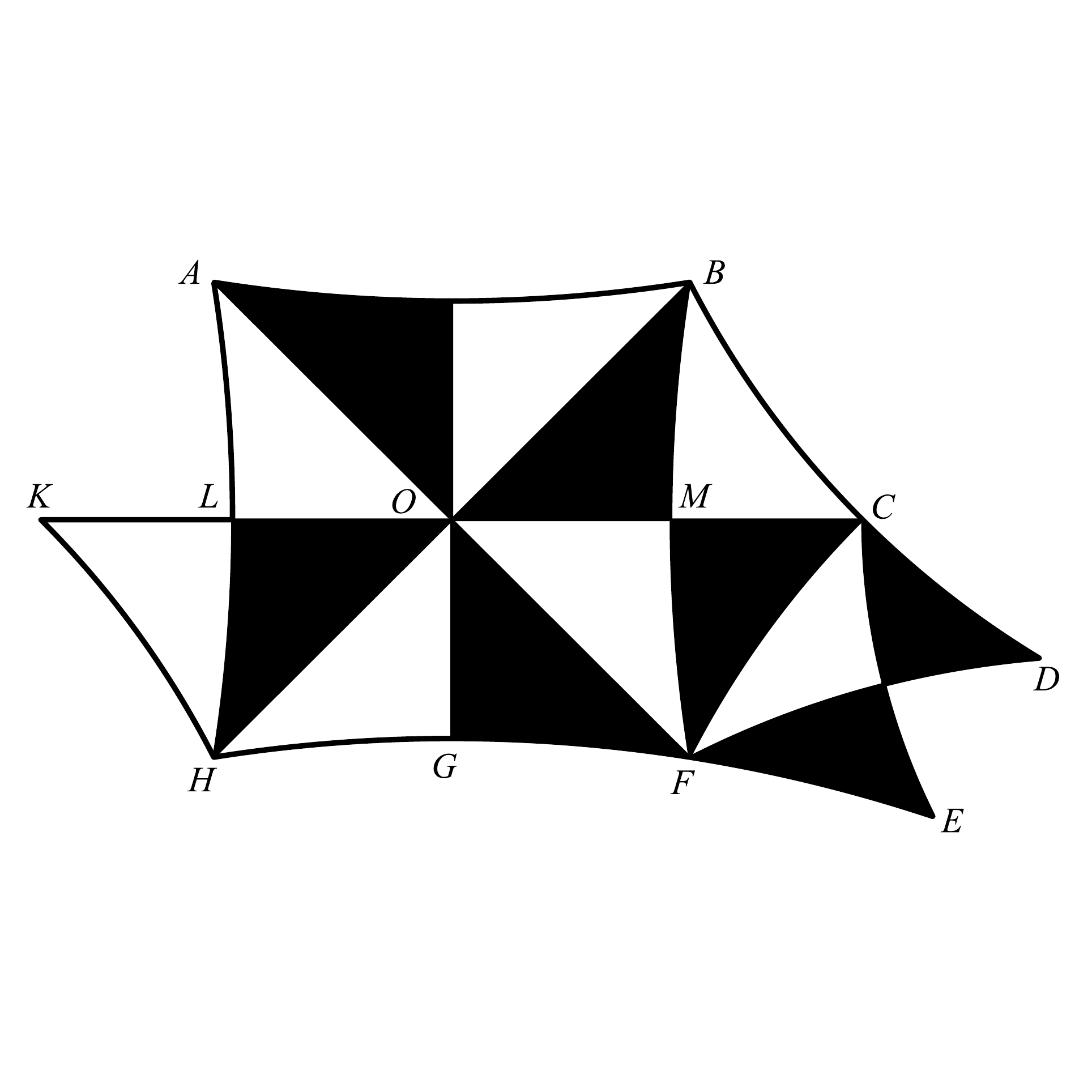}}
\put(-4,202){\includegraphics[height=102pt]{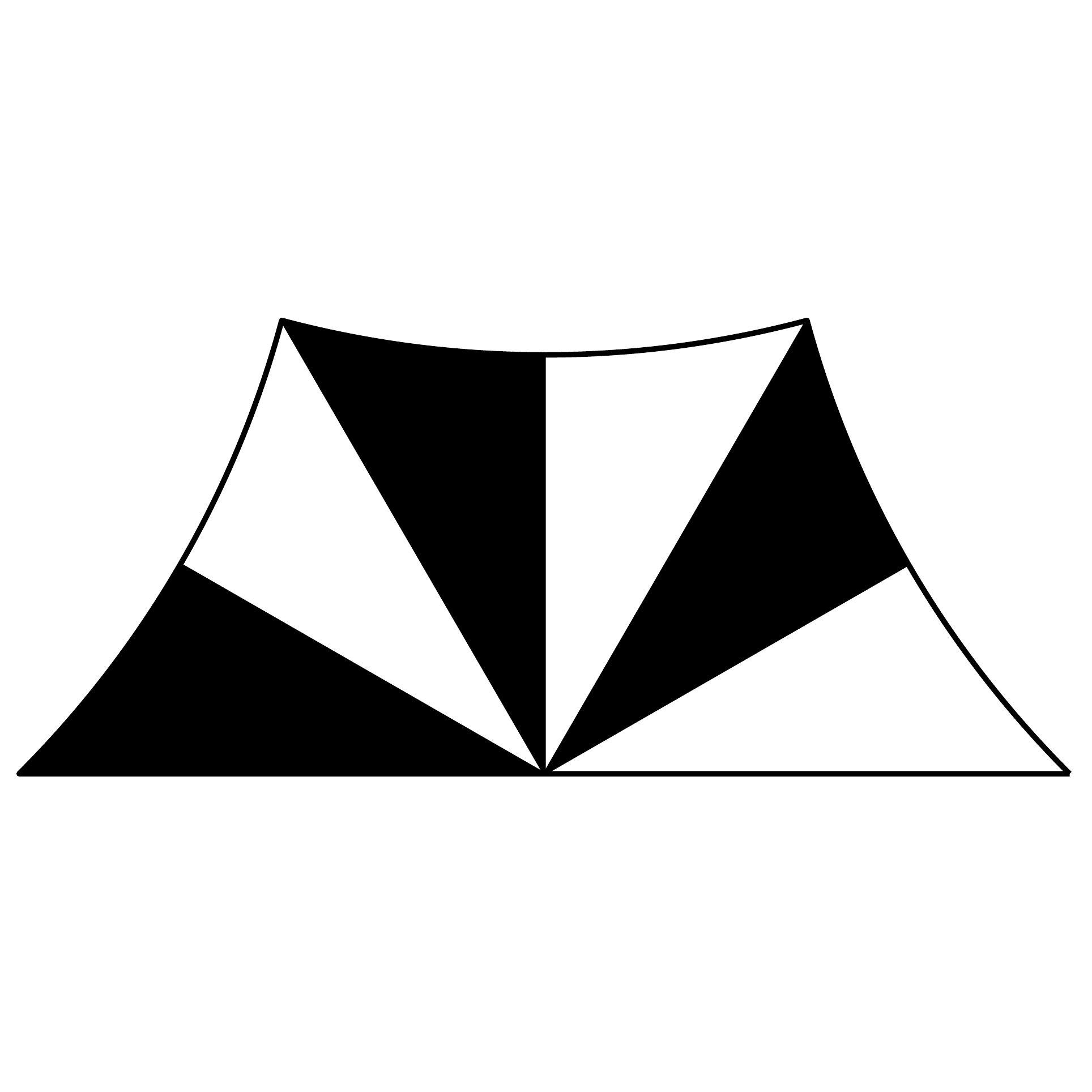}}
\put(-4,129){\includegraphics[height=90pt]{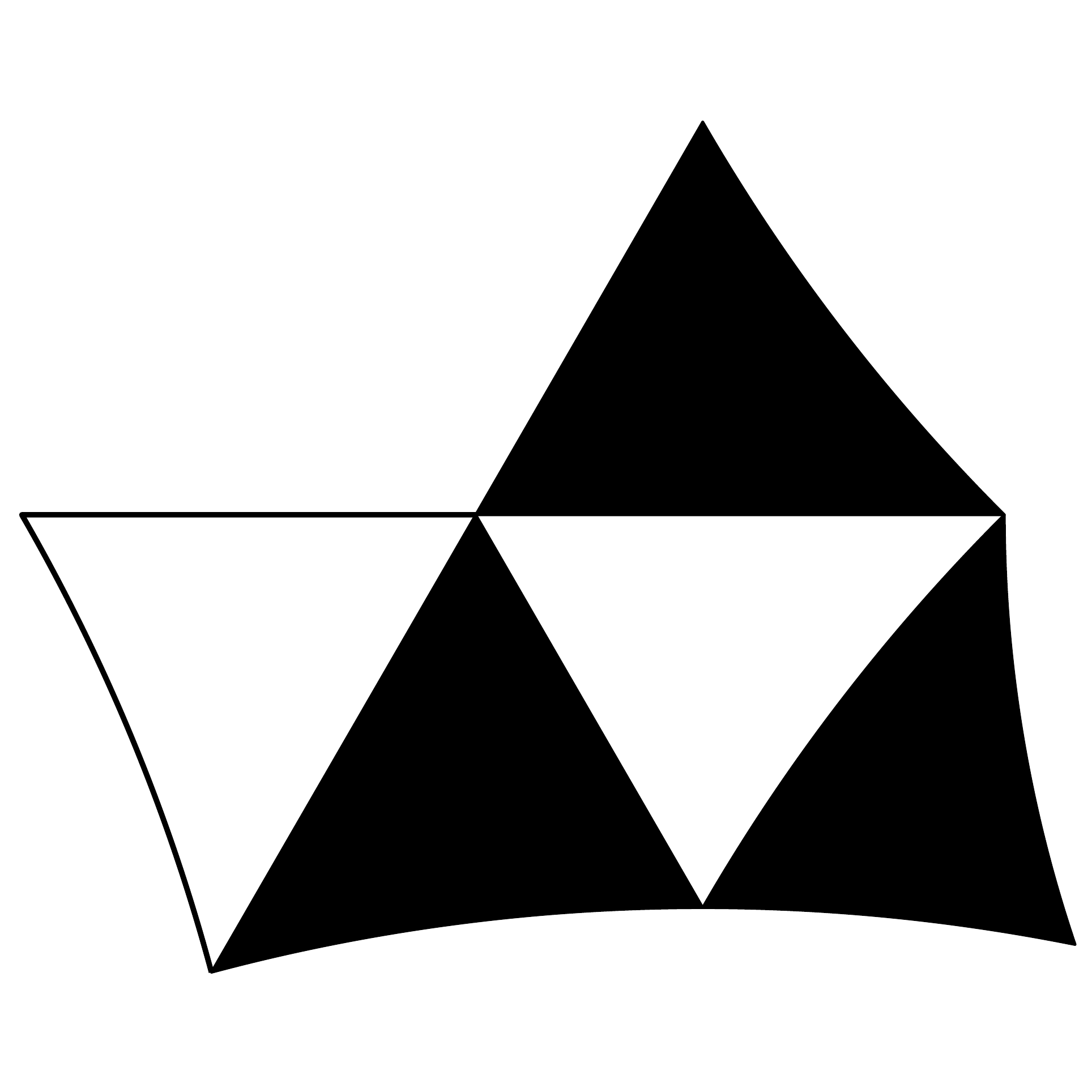}}
\put(-8,274){\refpart{a}}  \put(-8,202){\refpart{b}}
\put(-8,107){\refpart{c}} \put(178,132){\refpart{d}}
\end{picture}  \]
\caption{Coxeter decompositions for the parametric 
Gauss-to-Heun transformations defined over $\RR$}
\label{fg:felixon}
\end{figure}
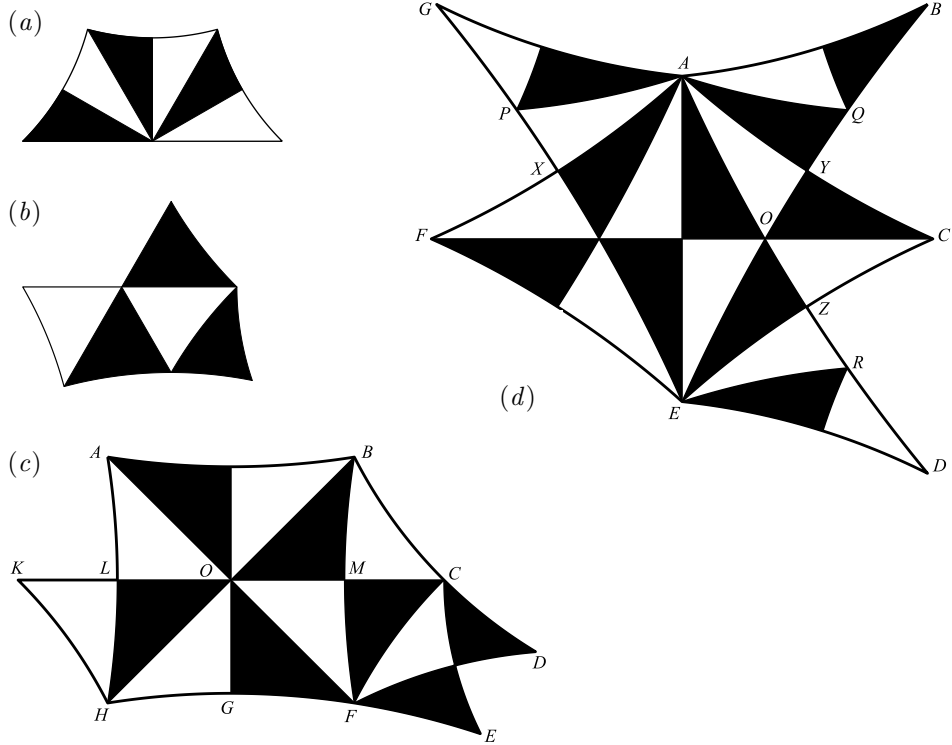

Similarly, Figure \ref{fg:felixon}\refpart{c} includes all Coxeter 
decompositions for the Gauss-to-Heun transformations of the types $(2,4)$ and $(2,5)$. 
Here we identify the quadrangles (and the corresponding Belyi coverings) for the Coxeter decompositions
\FE{6} to \FE{12}, respectively:
\begin{align*}
\mbox{\it ABCF }(\PH{25}), \quad \mbox{\it ABFH }(\PH{41}),  
\quad \mbox{\it ABDF }(\PH{20}), \quad \mbox{\it BDFH }(\PH{40}),  \\
 \mbox{\it ABML }(\PH{35}), \quad \mbox{\it ABCL }(\PH{29}), \quad \mbox{\it OCEG }(\PH{45}).
\end{align*}
The quadrangles (and coverings) for the Coxeter decompositions
\FE{14} and \FE{15} are 
\mbox{\it KCFH }(\PH{43}) and \mbox{\it OCEH }(\PH{24}), respectively.

Finally, Figure \ref{fg:felixon}\refpart{c} includes all Coxeter 
decompositions for the Gauss-to-Heun transformations specifically of the type $(2,3)$.
They are numbered from \FF{5} to \FF{27} by \cite[Figure 10]{Felikson98}. 
Here are their quadrangles (and coverings), respectively:
\begin{align*}
\mbox{\it AOEX }(\PH{30}), \quad \mbox{\it AXEZ }(\PH{28}),  \quad \mbox{\it AXEY }(\PH{27}), 
\quad \mbox{\it AFOY }(\PH{29}), \quad \mbox{\it AFOQ }(\PH{24}),  \\
\mbox{\it AOEP }(\PH{25}), \quad \mbox{\it AXEQ }(\PH{22}),  \quad \mbox{\it AXER }(\PH{23}), 
\quad \mbox{\it AFEO }(\PH{17}), \quad \mbox{\it AFOB }(\PH{15}),  \\
\mbox{\it APER }(\PH{20}), \quad \mbox{\it APEQ }(\PH{19}),  \quad \mbox{\it AOEG }(\PH{16}), 
\quad \mbox{\it AXED }(\PH{13}), \quad \mbox{\it APED }(\PH{9}), \; \\
\mbox{\it AXEB }(\PH{12}), \quad \; \mbox{\it APEB }(\PH{8}),  \; \quad \mbox{\it ACEF }(\PH{5}), 
\; \quad \mbox{\it ABEG }(\PH{2}) \;, \quad \mbox{\it ADEG }(\PH{4}), \; \\
\mbox{\it AXEC }(\PH{14}), \quad \mbox{\it APEC }(\PH{10}), \quad \mbox{\it ACEG }(\PH{3}). \; 
\end{align*}
In total, there are $(27-4)+(15-5)+(4-2)=35$ subdivisions \FF{k}, \FE{k}, \FT{k} 
representing Gauss-to-Heun transformations with exactly one parameter.

The subdivisions for the Gauss-to-Heun transformations with 2 or 3 parameters are
the following: 
\begin{itemize}
\item  The Coxeter decomposition for quadratic transformation (\ref{eq:quadtr2}) 
is represented by a single Schwarz reflection.
It can be discerned in many places in Figure \ref{fg:felixon}, for example as the quadrangle
\emph{OYCZ}  in picture \refpart{d}. 
It appears several times in Felikson's figures, in particular as \FF{1}=\FE{1}=\FT{1}.
\item There are two degree 3 decompositions \FF{2}=\FE{2} and \FT{2}. 
They are both represented by the covering $H_{34}$, as we mentioned discussing picture \refpart{a}.
The other cubic transformation (with the covering $H_{33}$) is not defined over $\RR$ in the normalized
form \cite[\S 4.4.4]{Vidunas2009b} but over $\QQ(\omega)$, hence there is no Coxeter decomposition for it.
\item There are three degree 4 decompositions, \FF3=\FE3, \FE4 and \FF4=\FE5.
They can be discerned, for example, as the following quadrangles in picture \refpart{c}, respectively:
{\it OBCF }(\PH{31}), {\it OABC }(\PH{47}),  {\it OCEF }(\PH{35}).
\end{itemize}

Whether a Gauss-to-Heun transformation is realized by a Coxeter decomposition,
is determined by a close inspection in Step 4 of \S \ref{sec:4}.  A necessary and sufficient condition
is that the Belyi covering has to be defined over $\RR$ after a normalization (by M\"obius transformations)
that locates 3 of the 4 singular points of Heuns equation as $x=0$, $x=1$, $x=\infty$. 
In particular, the fourth singular point $x=t$ has to be real, though this is not a sufficient condition.
For example, a proper normalization of $H_{48}$ for the type $(4,4)$ transformation
is $8ix(x^2-1)/(x+i)^4$. This gives $t=-1$, but the covering is not defined over $\RR$. 
There is one other example of this type: a proper normalization of $H_{28}$ for a type $(3,3)$ 
pull-back is $3(1+2\omega)x^2(x^2-1)/(x^2+\omega)^3$.
On the other hand,  a proper normalization of the same 
$H_{28}$ for a type $(2,3)$ pull-back is
defined over $\QQ(\sqrt{3})$, giving the Coxeter decomposition $F_6$. 
There are two different Coxeter decompositions for each of the following coverings:
\PH{20}, \PH{24}, \PH{25}, \PH{29}, \PH{34}, \PH{35}, \PH{47}. 
Comparison of our classification and Felikson's list \cite{Felikson98} provides 
a useful mutual confirmation.

The considered 
Coxeter decompositions are \emph{parametrized}, in~that one or more of the
triangular vertex angles are free to vary. For somewhat larger real values of the free parameter(s), 
the Coxeter decompositions 
are transfigured to spherical geometry  of the Riemann sphere 
(if angles larger than $\pi$ are allowed), as subdivisions of spherical 
quadrangles into  spherical triangles with the angles satisfying $\alpha\pi+\beta\pi+\gamma\pi>\pi$.
Most of the Coxeter decompositions can be transfigured to the plain Euclidean geometry 
(where $\alpha\pi+\beta\pi+\gamma\pi=\pi$) as~well. The exceptions are
\FF{14}, \FF{16}, \FF{20}, \FF{27}, \FE{9}, \FT{4}, for which the quadrangles degenerates to flat triangles.

Broughton et al.~\cite{Broughton2000} classify similar geometric objects:
\emph{divisible tilings} of the hyperbolic plane.  Compared with Felikson's
pictures, divisible tilings form a proper subset of Coxeter decompositions.
The condition for a Coxeter decomposition to be a divisible tiling is that
the quadrangle angles be equal to~$\pi/k$, with $k$ an integer.  In general
Coxeter decompositions, \emph{rational} multiples of~$\pi$ are also
allowed.  The one-parameter divisible tilings relevant here are depicted in
\cite[Table~6.6]{Broughton2000}.  There are 34 of them; the first~6
correspond to Gauss-to-Heun transformations with 2 or~3 parameters.  Divisible
tilings are indicated in Tables \ref{tab:clas}--\ref{tab:clas2} by the
notation $F_{7}^*,\dots,F_{34}^*$, where the subscripts refer to the
numbering in \cite[Table~6.6]{Broughton2000}.  There are $35-(34-6)=7$
relevant Coxeter decompositions with one parameter that are not divisible tilings;
they all have the angle $2\pi/3$. 

\subsection{Composite transformations}
\label{subsec:composite}

The composite Gauss-to-Heun transformations can be inductively deduced 
from a smaller set of pull-back transformations among hypergeometric and Heun functions.
Due to the associativity of the composition operation, one can always decompose a
Gauss-to-Heun transformation as a product of the following:
\begin{itemize}
\item A possibly composite Gauss-to-Gauss transformation 
  with a free parameter, excluding M\"obius transformations
  and from $\hpgde{1,\alpha,\alpha}$ or $\hpgde{1/2,1/2,\alpha}$ for the purposes of this article.  
  This could be  the quadratic transformation (\ref{eq:quadtr1}) and one of 6 classical transformations 
  (of degrees~$3$, $4$ and~$6$)  worked~out  by Goursat~\cite{Goursat1881}
  and listed in \cite[Table~1]{Vidunas2009}.
\item An \emph{indecomposable} Gauss-to-Heun transformation with at~least one
  free parameter.  This could be the quadratic transformation (\ref{eq:quadtr2});
  one of 4 other indecomposable transformations 
  of Table~\ref{tab:clas}; or an indecomposable transformation of
  Table \ref{tab:clas1} or~\ref{tab:clas2} of degree at most 6, possibly fitting a
  Gauss-to-Gauss or a Heun-to-Heun transformation. 
\item A Heun-to-Heun transformation with at~least one free parameter.  
  This could be the quadratic transformation~(\ref{eq:quadtr3}), 
  or the degree $4$ composite transformation 
\begin{align}
\label{eq:quadtr4}
\heunde{\fr12,\fr12,\fr12,\,\alpha}\pback{\HT2}
\heunde{\fr12,\fr12,\,\alpha,\,\alpha}\pback{\HT2}
\heunde{\alpha,\,\alpha,\,\alpha,\,\alpha},
\end{align}
realized by the covering \PH{31}. See \cite[\S 4.3]{Vidunas2009b} for an overview.
\end{itemize}

\begin{landscape}
{\small
\begin{figure}
\renewcommand{\arraystretch}{1.0}
\begin{center}
\begin{tabular}{cccccccccccccrccccccccccccl}
\cline{2-14}  \cline{20-26} & \multicolumn{13}{||c||}
{$\PHfortythree:\hpgde{\frac12,\frac13,\,p}\pback{3}\hpgde{\frac12,p,2p}\pback{2}\hpgde{p,p,4p}$}
&&&&&&
\multicolumn{3}{||r}{$\stackrel{}{\displaystyle\swarrow^{\hspace{-14pt}3}\hspace{10pt}}$} &
$\!\hpgde{\frac12,p,2p}\!$ & \multicolumn{3}{l||}{$\;\stackrel{}{\!\!\displaystyle\nwarrow^{\!2}}$} \\ 
\cline{2-16}  \multicolumn{2}{c|}{} && \multicolumn{13}{||c||}
{$\PH{35}:\hpgde{\frac12,\frac14,\,p}\pback{2}\hpgde{\frac12,p,p}\pback{2}\hpgde{p,p,2p}$}
&&&&  \multicolumn{2}{||l}{$\hpgde{\frac12,\,\frac13,\,p}$}
&&&& \multicolumn{2}{r||}{$\hpgde{2p,2p,2p}$} \\ 
\cline{4-17}  && \multicolumn{5}{|c|}{} &&&
\multicolumn{8}{||c||}{$\PH{33}:\hpgde{\frac13,\,\frac13,\,p}\pback{\CT3}\hpgde{p,\,p,\,p}$}
&&& \multicolumn{3}{||r}{$\stackrel{\displaystyle\nwarrow}{}^{\!2}$} & $\!\hpgde{\frac13,\frac13,2p}\!$
& \multicolumn{3}{l||}{$\;\stackrel{\displaystyle\swarrow^{\hspace{-17pt}\CT3}}{}$} \\ 
\cline{10-17}  \cline{20-26} \multicolumn{2}{r|}{\PH{2}} & \multicolumn{5}{r|}{\PHforty}
&&&&&&& \multicolumn{8}{|c|}{} \\ 
\multicolumn{2}{r|}{\PH{5}} & \multicolumn{5}{r|}{\PHfortyone} &
\multicolumn{6}{r|}{\PH{28}} & \multicolumn{8}{r|}{\PH{5}} \\  
&& \multicolumn{5}{|c|}{} &&&&&&& \multicolumn{8}{|c|}{} \\  
\cline{2-22} & \multicolumn{21}{|c|}{$\PH{32}:\hpgde{\alpha,\beta,\gamma}\pback{2}\heunde{\alpha,\alpha,2\beta,2\gamma}$}
& \multicolumn{5}{l}{$\alpha=\frac12\Rightarrow\PH{31}$} \\ 
\cline{2-22} \multicolumn{5}{c|}{} &&&&&&&&&& \multicolumn{5}{|c|}{}
&&&& \multicolumn{5}{l}{$\beta=\gamma=\frac14\Rightarrow\PHfortyeight$} \\ 
\multicolumn{5}{r|}{\PH{15}} & \multicolumn{9}{r}{\PH{35}\!}
& \multicolumn{5}{|r|}{\PH{27}\!} & \multicolumn{3}{l}{$\!\Rightarrow\PH{5}$}
& \multicolumn{5}{l}{$\alpha=\frac12,\beta=\frac14\Rightarrow\PH{31},\PH{35}\Rightarrow\PHfortyone$} \\ 
\multicolumn{5}{r|}{\PH{17}} & \multicolumn{9}{r}{\PH{31}\!}
& \multicolumn{5}{|l|}{$\!p=\frac14\Rightarrow\PHfortyone$} \\ 
\multicolumn{5}{r|}{\PH{19}} & \multicolumn{9}{r}{\PH{31}\!}
& \multicolumn{5}{|l|}{$\!q=\frac14\Rightarrow\PHfortyone$} \\ 
\multicolumn{5}{c|}{} &&&&&&&&&& \multicolumn{5}{|c|}{} \\ 
\cline{4-13} \cline{19-25} &&&
\multicolumn{10}{||c||}{$\PHfortyseven:\hpgde{\frac12,\,\frac13,\,p}\pback{4}\hpgde{\frac13,\,p,\,3p}$}
&& \multicolumn{4}{|c}{} &
\multicolumn{7}{||c||}{$\PH{34}:\hpgde{\frac12,\,\frac13,\,p}\pback{3}\hpgde{\frac12,\,p,\,2p}$} \\ 
\cline{4-13} \cline{19-25} &&&&&&
\multicolumn{6}{|r|}{} &&& \multicolumn{6}{|r|}{} & \multicolumn{4}{r|}{}\\ 
\multicolumn{6}{r|}{\PH{3}} & \multicolumn{6}{r|}{\PH{1}\!} &&& \multicolumn{6}{|r|}{\PH{3}}
& \multicolumn{4}{r|}{\PH{12}\!}\\ 
&&&&&& \multicolumn{6}{|r|}{\PH{6}\!} &&& \multicolumn{6}{|r|}{}
& \multicolumn{4}{r|}{\PH{14}\!}\\ 
\cline{2-11} \cline{17-23} &
\multicolumn{10}{|c|}{$\PH{34}:\hpgde{\frac13,\alpha,\beta}\pback{3}\heunde{\alpha,2\alpha,\beta,2\beta}$}
&& \multicolumn{2}{|c|}{} &&&
\multicolumn{7}{|c|}{$\PHfortyseven:\hpgde{\frac12,\alpha,\beta}\pback{4}\heunde{\alpha,3\alpha,\beta,3\beta}$}
& \multicolumn{1}{c|}{} \\ 
\cline{2-13} \cline{17-25}  \multicolumn{2}{c|}{} &&
\multicolumn{10}{|c|}{$\PH{33}:\hpgde{\frac13,\alpha,\beta}\pback{\CT3}\heunde{\alpha,\alpha,\alpha,3\beta}$}
&& \multicolumn{3}{|c|}{} &&
\multicolumn{7}{|c|}{$\PH{34}:\hpgde{\frac12,\alpha,\beta}\pback{3}\heunde{\frac12,\alpha,2\alpha,3\beta}$}
& \multicolumn{2}{l}{$\alpha=\frac14\Rightarrow\PHfortythree$} \\  
\cline{4-13} \cline{19-25} && \multicolumn{6}{|c|}{} &
\multicolumn{5}{r}{$\quad\alpha=\frac12\Rightarrow\PH{28}\Rightarrow\PH{6}\hspace{-5pt}$} & &
\multicolumn{3}{|c|}{} &&&&& \multicolumn{4}{|c}{} &
\multicolumn{2}{l}{$\beta=\frac16\Rightarrow\PHthirtynine$} \\ 
\multicolumn{2}{r|}{$p=\frac13:\PH{25}$} &
\multicolumn{6}{r|}{$p=\frac13:\PH{38}$} &&&&&&&
\multicolumn{3}{|r|}{$\;q=\frac14:\PH{20}\!$} &
\multicolumn{4}{r|}{$q=\frac14:\PH{25}\!$} \\ 
\multicolumn{2}{r|}{$q=\frac16:\PHthirtynine$} & \multicolumn{6}{r|}{$q=\frac16:\PH{38}$} &
\multicolumn{6}{c|}{} &&&& \multicolumn{4}{|r|}{} \\ 
&& \multicolumn{6}{|c|}{} & \multicolumn{6}{c|}{} &&&& \multicolumn{4}{|c|}{} \\ 
\cline{2-22} & \multicolumn{21}{||c||}{$\PH{32}:\hpgde{\frac12,\,p,\,q}\pback{2}\hpgde{p,\,p,\,2q}$} \\ 
\cline{2-22} &&&&&&& \multicolumn{8}{|c|}{} \\ 
\multicolumn{7}{r|}{$p=\frac13:\PH{19}$} &
\multicolumn{8}{r|}{$p=\frac13:\PH{20}$} \\  
\cline{21-26} &&&&&&& \multicolumn{8}{|c|}{} &&&&&&
\multicolumn{6}{|c|}{$\PH{26}:\hpgde{\frac12,\frac13,\alpha}\pback{6}
\heunde{\frac12,\frac12,\alpha,5\alpha}$} & $\Rightarrow\PH{4}$ \\ 
\cline{4-13} \cline{21-26} &&&
\multicolumn{10}{|c|}{$\PHfortysix:\hpgde{\frac13,\frac13,\alpha}\pback{\BT4}\heunde{\frac13,\frac13,\alpha,3\alpha}$}
& \multicolumn{2}{c|}{} &&&&&&
\multicolumn{6}{|c|}{$\PH{36}:\hpgde{\frac12,\frac13,\alpha}\pback{\AT4}
\heunde{\frac12,\frac12,\frac13,4\alpha}$} & $\Rightarrow\PH{20}$ \\ 
\cline{4-16} \cline{21-26} &&&&&
\multicolumn{11}{|c|}{$\PHfortyseven:\hpgde{\frac13,\frac13,\alpha}\pback{4}\heunde{\frac13,\frac13,2\alpha,2\alpha}$}
&&&&& \multicolumn{6}{|c|}{$\PH{36}:\hpgde{\frac12,\frac14,\alpha}\pback{\AT4}
\heunde{\frac12,\frac12,\alpha,3\alpha}$} & $\Rightarrow\PH{20}$ \\ 
\cline{6-16} \cline{21-26}
\end{tabular}
\caption{Compositions of pull-back transformations between
hypergeometric and Heun equations.}
\label{fig:comps} 
\end{center}
\end{figure}
}
\end{landscape}

Figure~\ref{fig:comps} graphically depicts all possible compositions of the
preceding three types (Gauss-to-Gauss, etc.)
The two longest boxes, centrally placed, represent quadratic transformations (based
on the double covering~$\PH{32}$).  The following objects and information are
included in the figure.


There are $7$ boxes with double edges on the left and the right
  sides, representing the classical Gauss-to-Gauss transformations.
  The quadratic transformation 
  appears as the long box in the lower part; two indecomposable
  transformations (of degree $3$ or~$4$) appear as boxes in the central
  part; and the remaining four classical transformations (of degrees $3$,
  $4$ and~$6$) are represented in the upper part.  Of~the latter, only the
  cubic transformation is indecomposable.  The transformation appearing
  near the upper-right corner can be decomposed in two different ways; its
  covering does not occur in Tables \ref{tab:clas}--\ref{tab:clas2}, hence
  it is not identified by an $H$~number. These $7$~boxes will be called
  $E\to E$ boxes.

The $10$ other boxes represent indecomposable Gauss-to-Heun
  transformations.  The quadratic transformation (\ref{eq:quadtr1}) is represented
  by the long box in the upper part; the four indecomposable
  transformations of Table~\ref{tab:clas} appear in the central part.  The
  three isolated boxes near the lower right corner represent the
  indecomposable transformations of Table~\ref{tab:clas2}, to each of which
  the quadratic Heun-to-Heun transformation~(\ref{eq:quadtr3}) can be applied.
  The other two lowest boxes represent transformations in
  Table~\ref{tab:clas2} that can be composed with a specialization of the
  quadratic $E\to E$ transformation.  These $10$~boxes will be called $E\to
  \textit{HE}$ boxes.

The vertical lines connect $E\to E$ and $E\to \textit{HE}$ boxes whose
  transformations can be composed (perhaps after a specialization of
  parameters).  The composed coverings are labeled by $H$~numbers on the
  left side of each vertical line.  Relevant specializations of the
  quadratic $E\to E$ transformation are given as~well.  Note that the
  specializations $p=\frac12$ and $q=\frac12$ of the quadratic $E\to E$
  transformation are not given, because (as~stated above) the dihedral 
  family is not considered here.  The number of possible compositions
  between an $E\to E$ box and an $E\to \textit{HE}$ box depends on the number of
  ways to identify (without degeneracy) the exponent differences of the
  intermediate hypergeometirc equation.  It is instructive to compare compositions of the
  quadratic $E\to \textit{HE}$ transformation with the two hypergeometric
  transformations coming from the coverings $\PHfortyseven$ and~$\PH{34}$.
  Compositions of the quadratic $E\to E$ and $E\to \textit{HE}$ transformations
  occur as the composite quartic coverings $\PH{35}$, $\PH{31}$ in
  Table~\ref{tab:clas}.

The $\Rightarrow$ symbols outside the boxes indicate application of
  the quadratic Heun-to-Heun transformation~(\ref{eq:quadtr3}). If this 
  transformation can be applied after an indecomposable Gauss-to-Heun
  transformation, the relevant parameter specializations and composite
  coverings are indicated to the right (or near the lower right corner) of
  the respective box.  If (\ref{eq:quadtr3}) can be applied after a
  composite Gauss-to-Heun transformation, this is indicated by the
  $\Rightarrow$~symbol to the right of the $H$~number of the composite
  covering (and to the right of the respective vertical line).

Some boxes of the same kind touch each other, but that does not have a
particular meaning.  The box for the quadratic $E\to \textit{HE}$ transformation
(\ref{eq:quadtr2}) is connected to all $E\to E$ boxes, since this 
transformation can always be applied without restrictions on the
exponent differences.  The box for the quadratic $E\to E$ transformation
(\ref{eq:quadtr1}) is connected to all $E\to \textit{HE}$ boxes, except for
the isolated three.

To show completeness of Figure~\ref{fig:comps}, one must:
\begin{itemize}
\item Consider all transformations of Tables
  \ref{tab:clas1},~\ref{tab:clas2} to which the quadratic Heun-to-Heun
  transformation~(\ref{eq:quadtr3}) can be applied; and after computing and
  examining the resulting compositions, keep only indecomposable
  transformations.
\item Check the classical $E\to E$ transformations of
  \cite[Table~1]{Vidunas2009} and the quadratic $E\to E$ transformations,
  to determine whether the pulled-back hypergeometric equation can ever
  have exponent differences of the form $1/k$,~$1/\ell$, consistent
  with~(\ref{eq:ledt}); if~so, composition with a one-parameter $E\to
  \textit{HE}$ transformation of Table~\ref{tab:clas1} or~\ref{tab:clas2}
  may be possible.
\item If a pair of $E\to E$ and $E\to \textit{HE}$ boxes is not connected
  by a vertical line, check that the respective transformations cannot be
  composed.
\item Check completeness of coverings for each vertical line.
\item Check possible compositions with the Heun-to-Heun transformations of
degrees $2$ and~$4$. 
\end{itemize}

The information of Figure~\ref{fig:comps} is given in the rightmost
columns of Tables~\ref{tab:clas}, \ref{tab:clas1}, \ref{tab:clas2}.
The compositions are spelled out more explicitly in \cite[Appendix B]{Vidunas2009b}.  
A~multiple occurrence of a covering in
Figure~\ref{fig:comps} means either that it can be decomposed in more than
one way (as~for \PH{3}, \PH{5}, \PH{6}, \PH{19}, \PH{31},
\PHthirtynine, \PHfortyone); or that it appears in more than one
composition (as~for \PH{25}, \PH{28}, \PH{35},~\PH{38}); or both (as
for~$\PH{20}$).  

The following cases are worth attention.  Firstly, there
are three ways to decompose the quartic covering~\PH{31} in
Table~\ref{tab:clas}, on account of its special branching pattern
\branch{2+2}{2+2}{2+2}:
\begin{equation}
\PH{31}\colon
\left\{
\begin{aligned}
\label{eq:compDD} 
& \hpgde{\fr12,\,\alpha,\,\beta}\pback{2}\hpgde{\alpha,\,\alpha,\,2\beta}
\pback{2}\heunde{2\alpha,\,2\alpha,\,2\beta,\,2\beta}, \\
& \hpgde{\fr12,\,\alpha,\,\beta}\pback{2}\hpgde{2\alpha,\,\beta,\,\beta}
\pback{2}\heunde{2\alpha,\,2\alpha,\,2\beta,\,2\beta}, \\
& \hpgde{\fr12,\,\alpha,\,\beta}\pback{2}\heunde{1/2,1/2,\,2\alpha,\,2\beta}
\pback{\HT2}\heunde{2\alpha,\,2\alpha,\,2\beta,\,2\beta}.
\end{aligned}
\right.
\end{equation}
This is indicated by the $2\times2$ in the rightmost column.  The covering
\PH{31} occurs as a part of the larger compositions \PH{5} and \PHfortyone;
see their composition lattices in \cite[(B.5), (B.4)]{Vidunas2009b}.
Besides, the covering \PH{31} induces the degree 4 Heun-to-Heun transformation~(\ref{eq:quadtr4}).

The transformation
$\hpgde{\fr12,\fr14,\,\alpha}\pback{4}\heunde{\fr12,\fr12,\,2\alpha,\,2\alpha}$
is induced by two distinct coverings: \PH{31} and~\PH{35}.  Induced
by~\PH{31}, this transformation is the $\beta=1/4$ specialization
of~(\ref{eq:compDD}); induced by~\PH{35}, this transformation is a new one
suggested by the branching pattern given in Table~\ref{tab:clas2}.
Both transformations have the factorization
\begin{equation}
\hpgde{\fr12,\,\fr14,\,\alpha}\pback{2}\heunde{1/2,1/2,1/2,\,2\alpha}
\pback{\HT2}\heunde{1/2,1/2,\,2\alpha,\,2\alpha},
\end{equation}
but they have different sets of $t$ parameters.
Both \PH{31} and \PH{35} appear as parts of the degree 8
composite transformation \PHfortyone.

\section{Existence and uniqueness of coverings}
\label{sec:6}

This section presents an elegant way to conclude that there are no Belyi coverings
with some branching patterns. The idea is to deduce a pull-back transformation
of Fuchsian equations that is not possible, because it would relate an equation
with finite monodromy to an equation with infinite monodromy group,
or the pulled-back equation would not exist. We apply this idea to all cases 
of non-existent coverings of Tables~\ref{tab:clas}, \ref{tab:clas1}, \ref{tab:clas2}.
Moreover, in \S \ref{subsec:BH} this approach is applied to most cases of non-existent 
coverings in the Miranda--Persson list \cite{Miranda89} 
of K3 elliptic surfaces.

As an immediate example, consider the non-existent covering of degree 4 in Table ~\ref{tab:clas}.
If it would exist, the specialization $\alpha=1/2$ would give a pull-back 
from  $\hpgde{1/2,1/2,\beta}$ to a Fuchsian equation with two singularities and (generally) non-equal
exponent differences  $\beta,3\beta$ at them, contradicting part \refpart{ii} of 
Lemma \ref{lem:logpoint} below.  Or one can further specialize $\beta=1$ or $\beta=1/3$ 
and get a contradiction  with part \refpart{i} of the same lemma. 
In \S \ref{subsec:principallemmas} we prove several assertions 
from which we make non-existence conclusions. 
Table \ref{tab:nonexist} outlines the non-existence proofs.
In~\S\,\ref{subsec:uniqueness}, we seek to show uniqueness (up to M\"obius transformations)
of the Belyi coverings with the encountered branching patterns,
by considering implied pull-backs between Fuchsian equations
with finite monodromy groups.

\subsection{Principal lemmas}
\label{subsec:principallemmas}

The easiest way to conclude non-existence of a Belyi covering with a certain branching
pattern is to deduce a pull-back transformation to a non-existent Fuchsian equation. 
Here are two basic situations.
\begin{lemma}
\label{lem:logpoint}
\begin{enumerate}
\item There is no Fuchsian equation on\/ $\PP^1$ that has exactly\/ one
  relevant singular point.
\item If a Fuchsian equation on\/ $\PP^1$ has exactly\/ $2$ singular points,
  their exponent differences are equal.
\end{enumerate}
\end{lemma}
\begin{proof}
If a Fuchsian equation has just one relevant singularity, we can move it to infinity
and make all points in $\CC$ ordinary. The differential equation then has the form
$y''+Py'+Qy=0$, where $P,Q$ are polynomials (in the differentiation variable $x$).
If $P=Q=0$, then the local exponents at the infinity are $0,-1$, thus
$x=\infty$ will be an irrelevant singularity. Otherwise $x=\infty$ is an irregular singularity, 
and the equation will not be Fuchsian.

If a Fuchsian equation has 2 singularities, we can assume them to be $x=0$, $x=\infty$.
The Liouville normal form of the equation is then $x^2y''=cy$ with $c\in\CC$.
The exponent differences of this equation equal $\sqrt{1+4c}$ at both singular points.
\end{proof}

Another type of non-existent transformation is a pull-back of
a hypergeometric equation with finite monodromy to a hypergeometric equation
with infinite monodromy.  (A~Fuchsian equation has finite monodromy if and only if its
solution space has a basis consisting of algebraic functions.)  The
following lemma characterizes some hypergeometric equations
with finite (or infinite) monodromy groups.

\begin{lemma}
\label{lem:finitem}
Consider a hypergeometric equation \/ $E=\hpgde{\alpha,\beta,\gamma}$ on\/ $\PP^1$.
\begin{enumerate}
\item Suppose that\/ $\alpha,\beta,\gamma$ are rational numbers, each
  having denominator\/ $3$.  Then the monodromy of\/ $E$ will be finite
  if and only if the sum of the numerators of\/ $\alpha,\beta,\gamma$ is even.
\item If\/ $\alpha$ is a half-odd-integer, and\/ $\beta,\gamma$ are
  rational numbers, each having denominator\/ $4$, then the monodromy 
  of\/ $E$ is not finite.
\item Suppose that\/ $\alpha,\beta,\gamma$ are integers.  Then the
  monodromy of\/ $E$ will be trivial if and only if the sum
  $\alpha+\beta+\gamma$ is odd, and the triangle inequalities
  $\gamma<\alpha+\beta$, $\beta<\alpha+\gamma$, $\alpha<\beta+\gamma$ are
  satisfied; otherwise the monodromy is not finite.
\item Suppose that\/ $\alpha$ is an integer while\/ $\beta,\gamma$ are
  half-odd-integers.  The set\/
  $\{\left|\beta-\gamma\right|,\allowbreak\beta+\gamma\}$ contains two
  integers of different parity; let\/ $k$ be the integer in this set such
  that\/ $k+\alpha$ is odd.  Then the monodromy group of\/ $E$ will be
  isomorphic to\/ $\ZZ/2\ZZ$ if and only if\/ $k<\alpha$; otherwise the
  monodromy will not be finite.
\end{enumerate}
\end{lemma}
\begin{proof}
We use the Schwarz classification of hypergeometric equations 
with finite monodromy for the first two statements; see \cite{Schwarz1872} or \cite[\S\,2.7.2]{Erdelyi53}.
The only possible projective monodromy in statement ~\refpart{a} is the
tetrahedral group~$\mathrak{A}_4$.  There are two Schwarz types (II and III) 
for this group: the hypergeometric equation must be contiguous either
to $\hpgde{1/2,1/3,1/3}$ or to $\hpgde{1/3,1/3,2/3}$. 
We must have the latter Schwarz type III. Contiguous relations shift the
exponent differences by integers whose sum is even. That does not change
the parity of the numerator sum (of the three integers divided by 3),
even if an exponent difference is multiplied by $-1$.

We do not find the denominator pattern of the statement \refpart{b} in the Schwarz list.
In particular, the two Schwarz types  (IV and V) for the octahedral group $\mathrak{S}_4$
are contiguous to $\hpgde{1/2,1/3,1/4}$ or $\hpgde{2/3,1/4,1/4}$. 

For the claim \refpart{c}, a representative solution of the generic hypergeometric equation with
trivial monodromy is $\hpgoppa21{-n,\,\ell+1}{-n-m}{z}$, with $n,m,\ell$
non-negative integers; see \cite[Theorem~2.4(5)]{Vidunas2007}.  Up~to a
permutation, one has 
that $\alpha=n+m+1$,
$\beta=n+\ell+1$, $\gamma=m+\ell+1$; that is
\begin{equation}
n=\frac{\alpha+\beta-\gamma-1}2,\qquad m=\frac{\alpha+\gamma-\beta-1}2,\qquad
\ell=\frac{\beta+\gamma-\alpha-1}2.
\end{equation}
If one of these three numbers is a negative integer, the singular point
with the largest exponent difference is logarithmic~\cite[\S\,9]{Vidunas2007}. 
If each of the above three numbers is a half-odd-integer, all three singular points are
logarithmic~\cite[\S\,5]{Vidunas2007}.

The assertion~\refpart{d} is a reformulation
of~\cite[Theorem~5.1]{Vidunas2008a}, stated in the context of hypergeometric
equations with either logarithmic solutions or the $\ZZ/2\ZZ$ monodromy group.
\end{proof}

Existence (and uniqueness) of coverings with a given
branching pattern can also be decided on the basis of transformations of
some hypergeometric equations with infinite monodromies.  
The following lemma implies that there are no transformations of $\hpgde{1/2,1/4,1/4}$
into itself of degrees 6, 12, 14, 21, 22, 24, or generally, of degrees $\equiv 3\pmod 4$,
even if suitable branching patterns of these degrees exist. 
Similarly, there are no transformations of $\hpgde{1/2,1/3,1/6}$ or $\hpgde{1/3,1/3,1/3}$
into themselves of degrees 6, 10, 15, 18, 22, 24, or generally, of degrees $\equiv 2\pmod 3$.
This lemma eludicates the non-uniqueness of \PH{44} and  \PH{21}.
\begin{lemma} 
\label{lem:elliptic}
\begin{enumerate}
\item Up to M\"obius transformations, the number of degree-$D$ pull-back
  coverings of $\hpgde{1/2,1/4,1/4}$ into itself is equal to the number of
  integer solutions $(a,b)$ with $a\ge 0$, $b>0$, of the equation
  $D=a^2+b^2$.
\item Up to M\"obius transformations, the number of degree-$D$ pull-back
  coverings of $\hpgde{1/2,1/3,1/6}$ or $\hpgde{1/3,1/3,1/3}$ into itself
  is equal to the number of integer solutions $(a,b)$ with $a\ge 0$, $b>a$,
  of the equation $D=a^2-ab+b^2$.
\end{enumerate}
\end{lemma}

\begin{proof}
According to \cite[\S\,8]{Vidunas2009}, the transformations of
$\hpgde{1/2,1/4,1/4}$ into itself correspond to isogenies of the 
$j=1728$ elliptic curve $y^2=x^3-x$.  The ring of isogenies
is isomorphic to the ring $\ZZ[i]$ of Gaussian integers, and the
degree of a pull-back is equal to the norm $a^2+b^2$ of the
corresponding $a+bi$.  In particular, the trivial and fractional-linear 
transformations correspond to the units $\pm1,\pm i$.  Therefore
one must count $a+bi\in \ZZ[i]$ such that $\left|a+bi\right|^2=D$ and $\arg(a+bi)\in [0,\pi/2)$.

Similarly \cite[\S\,8]{Vidunas2009}, the transformations of
$\hpgde{1/2,1/3,1/6}$ or $\hpgde{1/3,1/3,1/3}$ into itself correspond to
isogenies of the 
$j=0$ elliptic curves $y^2=x^3-1$ or $x^3+y^3=1$.  The
ring of isogenies is isomorphic to the ring of Eisenstein integers
$\ZZ[\omega]$.
The degree of a pull-back  is equal to the norm $a^2-ab+b^2$ of the corresponding
$a+b\omega$.  Trivial or M\"obius transformations correspond to the units
$\pm1,\pm\omega,\pm(\omega+1)$.  Therefore one must count
$a+b\omega\in\ZZ[\omega]$ such that $\left|a+b{\omega}\right|^2=D$ and
$\arg(a+b\omega)\in [0,\pi/3)$.
\end{proof}

\begin{table}
\begin{center}
{
\small
\begin{tabular}{cclcll} 
\hline  
Nonexistent  & Deg. & Branching pattern & Lemma & \multicolumn{2}{c}            {Exponent differences} \\ 
\cline{5-6} covering & $D$ & above singular points && 
\hfil hypergeom. & \hfil pulled-back \\
\hline
\PN{1} & 12 & \branch{\brep26}{\brep34}{7+3+1+1} &
\ref{lem:finitem}\refpart{a}&
$\led{\fr12,\fr13,\fr13}$ & $\led{\fr73,\,\fr13,\fr13}$ \\
\PN{2} &   & \branch{\brep26}{\brep34}{7+2+2+1} &
\ref{lem:logpoint}\refpart{b}&
$\led{\fr12,\fr13,\fr12}$ & $\led{\fr12,\,\fr72}$ \\
\PN{3} &   & \branch{\brep26}{\brep34}{6+4+1+1}  &
\ref{lem:finitem}\refpart{b} &
$\led{\fr12,\fr13,\fr14}$ & $\led{\fr32,\,\fr14,\fr14}$ \\
\PN{4} &   & \branch{\brep26}{\brep34}{6+2+2+2}  &
\ref{lem:logpoint}\refpart{a} &
$\led{\fr12,\fr13,\fr12}$ &  $\led{3}$ \\
\PN{5} &   & \branch{\brep26}{\brep34}{5+4+2+1} &
\ref{lem:finitem}\refpart{d} 
& $\led{\fr12,\fr13,\fr12}$ &  $\led{2,\,\fr12,\,\fr52}$ \\
\PN{6} &   &  \branch{\brep26}{\brep34}{5+3+3+1} &
\ref{lem:logpoint}\refpart{b} &
$\led{\fr12,\fr13,\fr13}$ & $\led{\fr13,\,\fr53}$ \\
\PN{7} &   &  \branch{\brep26}{\brep34}{5+3+2+2} &
\ref{lem:logpoint}\refpart{b}&
$\led{\fr12,\fr13,\fr12}$ & $\led{\fr32,\,\fr52}$ \\
\PN{8} &   &  \branch{\brep26}{\brep34}{4+4+3+1}  &
\ref{lem:logpoint}\refpart{b}&
$\led{\fr12,\fr13,\fr14}$ & $\led{\fr14,\,\fr34}$ \\
\PN{9} &   &  \branch{\brep26}{\brep34}{4+3+3+2} &
\ref{lem:logpoint}\refpart{b}&
$\led{\fr12,\fr13,\fr13}$ & $\led{\fr23,\,\fr43}$ \\
\PN{10} & 10 &  \branch{\brep25}{\brep33+1}{6+3+1}  &
\ref{lem:elliptic}\refpart{b} 
& $\led{\fr12,\fr13,\fr16}$ & $\led{\fr12,\fr13,\fr16}$ \\
\PN{11} &   & \branch{\brep25}{\brep33+1}{6+2+2} &
\ref{lem:logpoint}\refpart{b}&
$\led{\fr12,\fr13,\fr12}$ & $\led{\fr13,\,3}$ \\
\PN{12} &   & \branch{\brep25}{\brep33+1}{4+4+2} &
\ref{lem:logpoint}\refpart{b}&
$\led{\fr12,\fr13,\fr14}$  & $\led{\fr13,\,\fr12}$ \\
\PN{13} &   & \branch{\brep25}{\brep33+1}{4+3+3} &
\ref{lem:logpoint}\refpart{b}&
$\led{\fr12,\fr13,\fr13}$ & $\led{\fr13,\,\fr43}$ \\
\PN{14} & 9 & \branch{\brep24+1}{\brep33}{5+2+2} &
\ref{lem:logpoint}\refpart{b}&
$\led{\fr12,\fr13,\fr12}$ & $\led{\fr12,\,\fr52}$ \\
\PN{15}  &   & \branch{\brep24+1}{\brep33}{4+4+1} &
\ref{lem:logpoint}\refpart{b}&
$\led{\fr12,\fr13,\fr14}$ & $\led{\fr12,\,\fr14}$ \\
\PN{16} &  & \branch{\brep24+1}{\brep33}{3+3+3} & \ref{lem:logpoint}\refpart{a} &
$\led{\fr12,\fr13,\fr13}$ & $\led{\fr12}$ \\
\PN{17} & 8 & \branch{\brep24}{\brep32+2}{4+3+1} & \ref{lem:finitem}\refpart{a} &
$\led{\fr12,\fr13,\fr13}$ & $\led{\fr13,\fr23,\,\fr43}$ \\
\PN{18} &  & \branch{\brep24}{\brep32+2}{4+2+2}  & \ref{lem:logpoint}\refpart{b} &
$\led{\fr12,\fr13,\fr12}$ & $\led{2,\,\fr23}$ \\
\PN{19}  &  & \branch{\brep24}{\brep32+1+1}{5+3}  & \ref{lem:finitem}\refpart{a} &
$\led{\fr12,\fr13,\fr13}$ & $\led{\fr13,\fr13,\,\fr53}$ \\
\PN{20}& 6 & \branch{\brep23}{\brep31+2+1}{3+3} & \ref{lem:logpoint}\refpart{b} &
$\led{\fr12,\fr13,\fr13}$ & $\led{\fr13,\fr23}$ \\
\hline 
\PN{21} & 8 & \branch{\brep24}{\brep42}{5+1+1+1} & \ref{lem:finitem}\refpart{c} &
$\led{\fr12,\fr12,\,1}$ & $\led{2,\,2,\,5}$ \\
\PN{22} &   & \branch{\brep24}{\brep42}{3+2+2+1} & \ref{lem:logpoint}\refpart{b} &
$\led{\fr12,\fr14,\fr12}$ & $\led{\fr12,\,\fr32}$ \\
\PN{23} & 6 & \branch{\brep23}{\brep41+2}{4+1+1} & \ref{lem:logpoint}\refpart{b} &
$\led{\fr12,\fr12,\,1}$  & $\led{4,\,2}$ \\
\PN{24} &   & \branch{\brep23}{\brep41+2}{2+2+2} & \ref{lem:logpoint}\refpart{a} &
$\led{\fr12,\fr14,\fr12}$ & $\led{\fr12}$ \\
\hline 
\PN{25} & 6 & \branch{\brep23}{\brep51+1}{2+2+2} & \ref{lem:logpoint}\refpart{a} &
$\led{\fr12,\fr15,\fr12}$ & $\led{\fr15}$ \\
\hline 
\PN{26} & 6 & \branch{\brep32}{\brep32}{3+1+1+1} & \ref{lem:logpoint}\refpart{a} &
$\led{\fr13,\fr13,\,1}$ & $\led{3}$ \\
\hline 
\PN{27} & 4 & \branch{\brep22}{3+1}{2+2} & \ref{lem:logpoint}\refpart{a} &
 $\led{\fr12,\fr13,\fr12}$  & $\led{1/3}$ \\
 \hline
\end{tabular}
}
\bigskip
\caption{Unrealizable branching patterns, with a proof indication}
\label{tab:nonexist} 
\end{center} 
\end{table}

\subsection{Nonexistence of coverings}
\label{subsec:nonexist}

Tables \ref{tab:clas1}, \ref{tab:clas2} have 27 entries with nonexistent Belyi
coverings.  One branching pattern appears twice among the type $(2,4)$ candidates, 
hence the two tables actually have 26 different branching patterns 
with no covering. They are labeled $\PN{1},\dots,\PN{26}$. The repeating
branching pattern is labelled $\PN{23}$.
Nonexistence is in each case an immediate consequence of some lemma 
in \S \ref{subsec:principallemmas}.  Mostly by specialization of the free parameter, 
one either derives a pull-back from a hypergeometric equation 
to a nonexistent Fuchsian equation, or a pull-back of a hypergeometric equation
with finite monodromy to a hypergeometric equation with infinite monodromy, 
or a nonexistent pull-back of 
$\hpgde{\fr12,\fr13,\fr16}$ into itself. The unrealizable branching patterns and
the applicable lemmas are listed in Table~\ref{tab:nonexist}.

The non-existent covering of Table~\ref{tab:clas} is given the last number \PN{27}.
Its non-existence was already demonstrated at the beginning of this section.

Only for \PN{21} and \PN{23} the used implied transformation 
is not a specialization of a respective Gauss-to-Heun pull-back 
of  the classification in \S \ref{sec:4}. 
To prove \PN{21} by the specialization $\alpha=1/5$, 
one would need to inspect the 10 icosahedral Schwarz types in \cite[\S\,2.7.2]{Erdelyi53}.
The case \PN{23} can be proved using the specialization $\alpha=1/4$ of either
of the two candidate transformations in Table \ref{tab:clas2},
by invoking Lemma \ref{lem:elliptic}\refpart{b}.
Note that to use a hypergeometric equation with only two relevant singularities,
one must ensure that it is of the form $\hpgde{1,\alpha,\alpha}$.
In particular, Lemma \ref{lem:logpoint}\refpart{b} does not apply 
to the branching covering \brep26=\brep34=9+1+1+1 and 
its pull-backs from $\hpgde{1/2,1/3,1}$, because logarithmic singularities
rather than ordinary points appear.
And indeed, the covering $H_1$ exists.

\subsection{The Miranda-Persson classification}
\label{subsec:BH}


The lemmas of \S \ref{subsec:principallemmas}
can be applied to the problem of the existence of Belyi maps that 
would yield semi-stable elliptic fibrations of K3 surfaces with 6~singular fibers,
sorted out by Miranda, Person \cite{Miranda89} and Beukers, Montanus \cite{Beukers2008}. 
The degree of the relevant Belyi maps is 24, and their branching patterns
have the form \branch{\brep2{12}}{\brep38}{$P$}, 
where $P=a$+$b$+$c$+$d$+$e$+$f$ is a partition of~24 with exactly 6~parts.
There are 199 of these branching patterns in total. 
Miranda and Persson \cite{Miranda89} proved that Belyi coverings 
(and elliptic $K3$ surfaces) exist in 112 cases,
and do not exist in the remaining 87 cases. 
Beukers and Montanus ~\cite{Beukers2008} computed 
all\footnote{As pointed~out in the AMS MathSciNet review by David~P. Roberts, 
the table in \cite{Beukers2008} omits one Belyi covering for the
partition 10+6+4+2+1+1. Our computation confirms existence 
of two (rather than one) Belyi coverings for this partition:
\begin{align*}
& \frac{(144x^8+384x^7+1120x^6-784x^3+756x^2-240x+25)^3}
{108\,x^6\,(14x-5)^4\,(4x-1)^2\,(9x^2+24x+70)},\\
& \frac{(144x^8-1536x^7+5248x^6-5568x^5-720x^4+512x^3+192x^2+24x+1)^3}
{108\,(8x+1)^6\,x^4\,(x-3)^2\,(9x^2-42x-5)}.
\end{align*}
The second covering is missing in the Beukers--Montanus list. 
In total, there are 59 branching patterns (among the 112 indicated by Miranda and Persson)
with a unique Belyi map up to M\"obius transformations; 125 Galois orbits of the Belyi maps,
of size at most 4; and 191 different Belyi maps or {\em dessins d\'enfant}. }
those Belyi maps and checked non-existence for the 87 partitions. 

The non-existence proof in \cite{Miranda89} broadly relies on two techniques.
First, Miranda and Persson widen the space of considered branching patterns 
to include partitions $P$ with more than six 
parts\footnote{Therefore coverings with more than 3 branching fibers 
are allowed. Instead of the coverings, permutation representations 
of their monodromy are considered in \cite{Miranda89}. }
and conclude non-existence of coverings for a partition $a_1+\ldots+a_s$
from non-existence for a partition $a_1+\ldots+a_{s-1}+a_s'+a_s''$ with \mbox{$a_s=a_s'+a_s''$},
using \cite[Lemma (2.4)]{Miranda89}. 
Secondly, they get contradicting conclusions about the torsion of the assumed
elliptic surfaces in several non-existing cases.
In~\cite{Beukers2008}, non-existence is concluded either by using
a sum over the characters of $\mathrak{S}_{24}$ that counts coverings 
(not necessarily connected, with some rational weights) with a given branching pattern, 
or by direct computation. 
Let $\Sigma$ denote the counting character sum just mentioned,
given in \cite[Theorem 3.2]{Beukers2008}.
The large table in \cite{Beukers2008} does not list the 47
partitions  (out of the total 87) for which $\Sigma=0$.

Here we show that most of the non-existent cases in the Miranda--Persson list can be
deduced using the methods of \S \ref{subsec:principallemmas}. 
Here are 22 partitions out of the 40 ones with $\Sigma\neq0$ for which the non-existence 
can be proved by using Lemmas \ref{lem:logpoint}, \ref{lem:finitem},~\ref{lem:elliptic} directly:
\begin{center}
14+\brep25, 9+\brep35, 15+\brep24+1, 13+3+\brep24,
12+4+\brep24, 11+5+\brep24, 10+6+\brep24,\\
11+\brep34+1, 10+\brep34+2,  8+4+\brep34,
13+4+\brep23+1,  11+6+\brep23+1,  11+4+3+\brep23, 
10+4+4+\brep23,  9+8+\brep23+1, 9+6+3+\brep23,
8+7+3+\brep23, 8+5+5+\brep23, 7+7+4+\brep23, \\
10+\brep43+1+1, 6+\brep43+3+3, \brep63+3+2+1.
\end{center}
The choice of the starting $\hpgde{\fr12,\fr13,\fr1k}$ that yields a
non-existent covering is indicated by the \brep{k}{n} notation.
Next, here are 22 partitions out of the 47 ones with $\Sigma=0$ to which our lemmas apply directly:
\begin{center}
9+7+\brep24, 7+5+\brep34, 7+\brep44+1, 6+\brep44+2, 5+\brep44+3, \brep54+3+1, \\
9+5+4+\brep23, 7+6+5+\brep23, 
13+\brep33+1+1, 11+\brep33+2+2, 10+4+\brep33+1, 8+6+\brep33+1, 8+5+\brep33+2,
7+7+\brep33+1, 7+6+\brep33+2, 7+4+4+\brep33, 6+5+4+\brep33, 5+5+5+\brep33,
9+\brep43+2+1, 6+5+\brep43+1, 7+\brep43+3+2, 5+5+\brep43+2.
\end{center}
Additionally, the four cases 7+\brep53+1+1, 6+\brep53+2+1, \brep53+4+4+1, \brep53+4+3+2
with \mbox{$\Sigma=0$} are concluded by inspecting the icosahedral hypergeometric equations
in the Schwarz table  \cite[\S\,2.7.2]{Erdelyi53}. In total, this shows 48 out of the 87 cases.

More cases of non-existence can be deduced from implied pull-backs to Fuchsian equations
with 3 non-apparent singularities and a few apparent singularities. These equations are 
gauge ``contiguous" to target hypergeometric equations (with infinite or infinite monodromy)
as the local exponent differences differ at all points by integers. The total shift of the exponent differences,
including those from the difference 1 for ordinary points of hypergeometric equations, 
must be an even integer. In this way, non-existence for the following 7 partitions with $\Sigma\neq0$
can be shown:
\begin{center}
10+6+\brep32+1+1, 9+9+\brep31+1+1+1, 8+6+\brep32+2+2,
7+6+6+\brep31+1+1, 7+6+4+\brep32+1, 6+5+5+\brep32+2,  8+6+\brep42+1+1.
\end{center}
In each case, the apparent singularities are represented by the branching orders that are 
integer multiples of the bracketed numbers. And here are 7 partitions with $\Sigma=0$
that can be handled in the same way:
\begin{center}
9+7+\brep32+1+1,  9+5+\brep32+2+2, 9+4+4+\brep32+1, 
6+6+5+\brep31+2+2, 6+6+4+4+\brep31+1, 8+5+\brep42+2+1, 8+\brep42+3+3+2.
\end{center}
Besides, a pull-back from $\hpgde{1/2,1/3,1/3}$ can be applied 
to show  the non-existence for 9+6+6+1+1+1, with $\Sigma\neq0$. 
It is trickier to combine parts \refpart{c}, \refpart{d} of Lemma \ref{lem:finitem} with gauge shifts.

Of the remaining $87-48-7-7-1=24$ partitions, the following 6 (with $\Sigma\neq0$)
and 11 (with $\Sigma=0$) partitions could be handled with a full knowledge 
of Heun equations  with finite monodromy (that are not classified yet):
\begin{center}
10+8+\brep22+1+1, 13+\brep42+1+1+1, 11+\brep42+2+2+1, 
9+\brep42+3+2+2, 9+\brep52+3+1+1, 
8+\brep52+4+1+1; \\ 
9+6+4+\brep22+1, 8+8+3+\brep22+1, 7+7+5+\brep22+1, 7+6+4+3+\brep22,  \\
10+5+\brep32+2+1, 7+5+5+\brep32+1, 9+5+\brep42+1+1, 7+5+\brep42+2+2,  \\
 8+\brep52+3+2+1, 6+\brep52+4+3+1, 6+\brep52+4+2+2.
\end{center}
Besides, a pull-back from $\hpgde{1/2,1/3,1/4}$ could be then applied to two partitions
with $\Sigma=0$: 12+8+1+1+1+1, 8+8+5+1+1+1.
Other 3 partitions  (with $\Sigma\neq0$)
\begin{center}
12+5+\brep41+1+1+1, 10+\brep51+4+3+1+1, 
9+8+\brep41+1+1+1,   
\end{center}
could be decided by Fuchsian equations with 4+1 singularities 
(i.e., 4 non-apparent and 1 apparent).
There remain only two partitions: 7+7+6+2+1+1 with $\Sigma\neq 0$, and 
7+7+4+3+2+1 with $\Sigma=0$. Their non-existence might be decided by using implied pull-backs 
from $\hpgde{1/2,1/3,1/2}$ to Fuchsian equations with 4+1 singularities and the monodromy 
group $\mathrak{D}_2$ or $\ZZ/2\ZZ$.

\subsection{Uniqueness of coverings}
\label{subsec:uniqueness}

Uniqueness of Gauss-to-Heun transformations (and of their coverings)
with a plausible branching pattern can be concluded from uniqueness of 
specialized Gauss-to-Gauss transformations.
In particular, the coverings \PH{32}\,--\PH{35}, \PH{43}, \PH{47} appear in the classical
hypergeometric transformations listed by Goursat \cite{Goursat1881}.
The coverings \PH{1}, \PH{2}, \PH{7}, \PH{8}, \PH{11}, \PH{18}, \PH{42} 
appear in the hypergeometric transformations from $\hpgde{k,\ell,m}$ with $k,\ell,m$
positive integers satisfying $1/k+1/\ell+1/m<1$. As determined in \cite{Vidunas2005}
(and \cite[\S 9]{Vidunas2009}), these pull-backs are unique up to M\"obius transformations as well.
The coverings \PH{31}, \PH{39}, \PH{41}, \PH{45} apply to hypergeometric transformations
from $\hpgde{1/2,1/2,\alpha}$ with infinite dihedral monodromy \cite[\S 4]{Vidunas2012}. 
The pulled-back equations have infinite cyclic or dihedral monodromy. 
They are, respectively,
\begin{equation*}
\hpgde{1,2\alpha,2\alpha}, \quad \hpgde{1/2,1/2,6\alpha}, \quad
\hpgde{1,4\alpha,4\alpha}, \quad \hpgde{1/2,1/2,5\alpha}.
\end{equation*}
The cyclic covering \PH{48} gives the pull-back
$\hpgde{1,\alpha,\alpha}\pback{4}\hpgde{1,4\alpha,4\alpha}$ of hypergeometric equations
with infinite cyclic monodromy. 

Non-unique Gauss-to-Gauss transformations appear when hypergeometric equations
$\hpgde{k,\ell,m}$ are pulled-back, with $k,\ell,m$ positive integers satisfying $1/k+1/\ell+1/m\ge 1$. 
It the equality holds, these hypergeometric functions are integrals of holomorphic differentials
on $j=1728$ or $j=0$ elliptic curves \cite[\S 8]{Vidunas2009}. 
Lemma \ref{lem:elliptic} counts the coverings \PH{3}, \PH{12}, \PH{21}, \PH{40}, \PH{44}, \PH{46}.
If can be established (by identifying transformations of holomorphic differentials
on the curves $y^2=x^3-1$ and $x^3+y^3=1$, $y^2=x^6+1$) that the transformations from 
$\hpgde{1/2,1/3,1/6}$  to $\hpgde{1/3,1/3,1/3}$ or $\hpgde{2/3,1/6,1/6}$ are compositions 
of the pull-backs of Lemma  \ref{lem:elliptic} with quadratic transformations. 
This applies to the coverings \PH{15}, \PH{19}, \PH{38}.

The hypergeometric equations $\hpgde{k,\ell,m}$ with $1/k+1/\ell+1/m>1$
have finite monodromy groups. The hypergeometric solutions are thereby algebraic functions.
These equations play a fundamental role in the classical heory of algebraic solutions
of second order Fuchsian equations:
\begin{itemize}
\item $\hpgde{1,\,1/k,\,1/k}$, with the finite cyclic monodromy $\mathrak{C}_k$.
\item $\hpgde{1/2,\,1/2,\,1/k}$, with the dihedral projective monodromy $\mathrak{D}_k$.
\item $\hpgde{1/2,\,1/3,\,1/3}$, with the tetrahedral projective monodromy $\mathrak{A}_4$.
\item $\hpgde{1/2,\,1/3,\,1/4}$, with the octahedral projective monodromy $\mathrak{S}_4$.
\item $\hpgde{1/2,\,1/3,\,1/5}$, with the icosahedral projective monodromy $\mathrak{A}_5$.
\end{itemize}
By a celebrated theorem of Klein~\cite{Klein1877}, 
all second order Fuchsian equations on~$\PP^1$ with a finite monodromy group are pull-backs
of one of these standard hypergeometric equations, with the same projective monodromy group. 
These {\em Klein transformations} are known to be unique up to M\"obius transformations
~\cite{Baldassarri79}. However, pull-back transformations between hypergeometric equations
with different projective monodromy need not to be unique. 
Litcanu \cite[Theorem 2.1]{Litcanu2004} noted non-uniqueness of the pull-backs
from $\hpgde{1/2,1/3,1/4}$ to $\hpgde{1/2,1/2,1/2}$ and $\hpgde{1,1/2,1/2}$, 
of degree 6 and 12 respectively.
The non-uniqueness is caused by pairs of different branching patterns though,
e.g.,~2+2+2=3+3=2+2+2 and 2+2+1+1=3+3=4+2. 
The example of  $\hpgde{1/2,1/2,1/5}\pback{10}\hpgde{1/2,1/2,2}$ in \cite[\S 5.4]{Vidunas2012}
shows that non-unique coverings with the same branching pattern easily occur 
for pull-backs to equations with apparent singularities.
Besides, many compositions of
\begin{equation} \label{eq:ico2tet}
\PH{37}: \hpgde{1/2,1/3,1/5} \pback5 \hpgde{1/2,1/3,1/3}
\end{equation}
with transformations from the tetrahedral equation are not unique either, 
because the properly normalized $H_{37}$ is defined over $\QQ(\sqrt{-15})$;
see formula \cite[(50)]{Vidunas2009}. 

In Table \ref{tab:coverings}, the coverings
\PH{9}, \PH{10}, \PH{13}, \PH{16}, \PH{22}, \PH{24} give Klein transformations
of $\hpgde{1/2,\,1/3,\,1/5}$ to the following hypergeometric equations,
respectively:
\begin{align*}
\hpgde{1/3,1/5,4/5},\quad \hpgde{1/3,2/5,3/5},\quad \hpgde{1/2,1/5,3/5},\\
\hpgde{2/3,1/5,2/5},\quad \hpgde{1/2,1/3,2/5},\quad \hpgde{1/3,2/3,1/5}.
\end{align*} 
This illustrates the Schwarz types VIII, XV, IX, X, XIV, XII, respectively. 
The other icosahedral Schwarz types are represented 
by $\hpgde{1/3,1/3,2/5}$, $\hpgde{1/5,1/5,4/5}$, $\hpgde{2/5,2/5,2/5}$, 
and the standard $\hpgde{1/2,1/3,1/5}$. 
Uniqueness of the coverings \PH{14}, \PH{17}, \PH{23}, \PH{29}, \PH{30}
is established by noting these Klein transformations: 
\begin{align*}
\hpgde{1/2,1/3,1/3}\pback{9} \hpgde{1/2,2/3,4/3},\quad 
\hpgde{1/2,1/2,1/3}\pback{8} \hpgde{3/2,3/2,2/3},\\
\hpgde{1/2,1/3,1/4}\pback{7} \hpgde{1/2,1/3,3/4},\quad
\hpgde{1/2,1/3,1/4}\pback{5} \hpgde{1/2,2/3,1/4},\\ 
\hpgde{1/2,1/3,1/3}\pback{5} \hpgde{1/2,2/3,2/3}.
\end{align*} 

These considerations of reduction to hypergeometric transformations do not
immediately establish uniqueness of 10 coverings in Table \ref{tab:coverings}.
Those coverings induce rather attractive transformations between hypergeometric equations
with different finite monodromy. In particular, \PH{6}, \PH{28} pull-back 
$\hpgde{1/2,1/3,1/3}$ 
to $\hpgde{1,1,1}$ and $\hpgde{1,1/2,1/2}$; 
then \PH{5}, \PH{20}, \PH{25}, \PH{27}, \PH{36} transform 
$\hpgde{1/2,1/3,1/4}$ to $\hpgde{1,1/2,1/2}$, $\hpgde{1,1/3,1/3}$, $\hpgde{1/2,1/3,2/3}$,
$\hpgde{1/2,1/2,1/2}$, $\hpgde{1/2,1/2,1/3}$, respectively;
and finally, \PH{4}, \PH{26}, \PH{37} pull-back $\hpgde{1/2,1/3,1/5}$ to 
$\hpgde{1,1/5,1/5}$, \mbox{$\hpgde{1/2,1/2,1/5}$} and (\ref{eq:ico2tet}).
Many of the coverings pull-back $\hpgde{1/2,1/3,1/2}$ or other 
dihedral hypergeometric equations to hypergeometric
equations with simpler dihedral or cyclic monodromy.
 
\begin{acknowledgement}
The authors are very grateful to Robert S. Maier for sharing his knowledge of
literature and ongoing developments related to the subject of this article, and
a coordination discussion.

R.V.~ is supported by  supported by the JSPS grant No 20740075.
Calculations by G.F. were obtained in the Interdisciplinary Centre for
Mathematical and Computational Modelling (ICM), Warsaw University, under
grant no.\ G34-18.  R.V. and~G.F. are grateful for the hospitality provided
by IMPAN and the organizers of the XVth Conference on Analytic Functions
and Related Topics, held in Chelm in July 2009.
\end{acknowledgement}




\end{document}